\documentclass[a4paper]{article}

\usepackage{amsmath,amsfonts,amssymb,amsthm}
\usepackage{url}
\usepackage{pgf,hyperref}
\usepackage{xspace}
\usepackage{tikz-cd}

\usepackage{pgfplots}
\pgfplotsset{compat=newest}
\usepgfplotslibrary{fillbetween}

\DeclareMathOperator{\im}{im} % like \ker
\DeclareMathOperator{\e}{e} % like \ln
\DeclareMathOperator{\diag}{diag}
\DeclareMathOperator{\rank}{rank}
\DeclareMathOperator{\innt}{int}
\DeclareMathOperator{\lin}{lineal}
\DeclareMathOperator{\rec}{rec}
\newcommand{\trans}{\mathsf{T}}
\newcommand{\pol}{\mathrm{pol}}

\newcommand{\irr}{core\xspace}

\newcommand{\EE}{\mathcal{E}}
\newcommand{\Ac}{\mathcal{A}_{k,\EE}}
\newcommand{\ST}{\mathcal{S}_{k,\EE}}
\newcommand{\CE}{\mathcal{C}_\EE}
\newcommand{\PE}{\mathcal{P}_{k,\EE}}

\newcommand{\la}{\lambda}
\newcommand{\Gkl}{G_k^{\la}}
\newcommand{\Vl}{{V^{\la}}}
\newcommand{\El}{{E^{\la}}}

\newcommand{\kl}{k^{\la}}
\newcommand{\Akl}{A_k^{\la}}
\newcommand{\Kkl}{K_k^{\la}}
\newcommand{\GEEl}{G_\EE^{\la}}
\newcommand{\EEl}{{\EE^{\la}}}
\newcommand{\IEEl}{I_\EE^{\la}}

\newcommand{\R}{{\mathbb R}}

\newcommand{\dd}[2]{\frac{\mathrm{d} #1}{\mathrm{d} #2}}

\theoremstyle{plain}
\newtheorem{theorem}{Theorem}%[section]
\newtheorem{proposition}[theorem]{Proposition}
\newtheorem{lemma}[theorem]{Lemma}
\newtheorem{fact}[theorem]{Fact}

\theoremstyle{definition}

\newtheorem{remark}[theorem]{Remark}

\newcommand\blue[1]{\textcolor{black}{#1}}

\makeatletter
\def\blfootnote{\xdef\@thefnmark{}\@footnotetext}
\makeatother

\begin{document}

\title{A new decomposition of the graph Laplacian \\ and the binomial structure of mass-action systems}

\renewcommand{\thefootnote}{\fnsymbol{footnote}}

\author{% 
Stefan M\"uller$$
%\footnotemark[1]
}
\blfootnote{
\scriptsize

\noindent
{\bf S.~M\"uller} (\href{mailto:st.mueller@univie.ac.at}{st.mueller@univie.ac.at}),
Faculty of Mathematics, University of Vienna, Oskar-Morgenstern-Platz 1, 1090 Wien, Austria

%\smallskip
%\noindent
%$^*$ Corresponding author
}

\date{\today \\ \bigskip \bigskip 
\normalsize \em 
Dedicated to the memory of Friedrich J.\ M.\ Horn
on the occasion of the 50th anniversary of his foundational 1972 papers
on mass-action systems}

\maketitle

\begin{abstract}
We provide a new decomposition of the Laplacian matrix
(for labeled directed graphs \blue{with strongly connected components}),
involving an \blue{invertible {\em core matrix},}
the vector of tree constants,
and the incidence matrix of an auxiliary graph, representing an order on the vertices.
Depending on the particular order,
the core matrix has additional properties.
Our results are graph-theoretic/algebraic in nature.

As a first application,
we further clarify the binomial structure of (weakly reversible) mass-action systems,
arising from chemical reaction networks.
Second, we extend a classical result by Horn and Jackson
on the asymptotic stability of special steady states (complex-balanced equilibria).
\blue{Here,} the new decomposition of the graph Laplacian allows \blue{us} to consider
regions in the positive orthant with given {\em monomial evaluation orders}
(and corresponding polyhedral cones in logarithmic coordinates).
As it turns out, all dynamical systems are asymptotically stable that can be embedded in certain {\em binomial differential inclusions}.
In particular, this holds for complex-balanced mass-action systems,
and hence we also obtain a polyhedral-geometry proof of the classical result.


%\medskip
%{\bf Mathematics Subject Classification (2010): 1, 2, 3.} 
\end{abstract}

% ========= ========= ========= ========= ========= ========= ========= =========

\section{Introduction} \label{sec:intro}

The Laplacian matrix (or graph Laplacian) is a matrix representation of a graph.
It can be seen as a discrete version of the Laplace operator defined on graphs.
On the one hand,
the Laplacian matrix of an {\em undirected} graph, its spectrum, and its eigendecomposition
have a variety of applications ranging from organic chemistry to signal processing and machine learning~\cite{Mohar1991,Merris1994,Shuman2013,BelkinNiyogi2003}.
On the other hand,
{\em labeled, directed} graphs underlie dynamical systems
ranging from continuous-time Markov processes (linear stochastic models)~\cite{MirzaevGunawardena2013}
to mass-action systems (non-linear deterministic models of chemical reaction networks)~\cite{HornJackson1972}.

In the linear setting, the vertices $V$ of a simple digraph $G=(V,E)$ represent states, and the edges $E$ represent transitions.
Moreover, edge labels $k$ represent transition rate constants.
The dynamical system for a state variable $\psi$ is given by 
\begin{equation}
\dd{\psi}{t} = A_k \, \psi ,
\end{equation}
where $A_k$ is the Laplacian matrix of the labeled digraph $G_k=(V,E,k)$. 
That is, $(A_k)_{i,j} = k_{j \to i}$ if there is a transition $(j \to i) \in E$,
$(A_k)_{i,i} = -\sum_{(i \to j) \in E} k_{i \to j}$, and $(A_k)_{i,j}=0$ otherwise.
(As in chemical reaction network theory, we use the letter $A$ for the graph Laplacian
and indicate its dependence on the edge labels $k$ by a subscript.)
The linear system can be called ``Laplacian dynamics'',
it is equivalent to the stochastic master equation,
and it is studied in applications ranging from biochemistry to systems biology~\cite{Gunawardena2012,MirzaevGunawardena2013}.

\blue{
In the nonlinear setting,
the dynamical system for the species concentrations $x$ is given by 
\begin{equation} \label{eq:intro}
\dd{x}{t} = Y A_k \, x^Y .
\end{equation}
All notation is defined at the end of this introduction,
and mass-action systems are introduced in Section~\ref{sec:mas}.
Here, we motivate Eqn.~\eqref{eq:intro} in an informal way.
As an example, we consider the chemical reaction ${1 \mathsf{X}_1+ 1 \mathsf{X}_2 \to \mathsf{X}_3}$
with ``stoichiometric'' coefficients equal to 1.
Under the assumption of mass-action kinetics,
its rate is given by $v = k \, (x_1)^1 (x_2)^1$, 
where $k > 0$ is the rate constant, and $x_1, x_2 \ge 0$ are the concentrations of the species $\mathsf{X}_1,\mathsf{X}_2$.
More abstractly,
we can write the reaction as ${y \to y'}$ with (educt and product) ``complexes'' $y = (1,1,0,0,\ldots)^\trans$ and $y'=(0,0,1,0,\ldots)^\trans$,
and we can write its rate as $v = k \, x^y$ with the monomial $x^y := \prod_j (x_j)^{y_j}= (x_1)^1 (x_2)^1 (x_3)^0 (x_4)^0 \cdots$
in the species concentrations~$x=(x_1,x_2,x_3,x_4,\ldots)^\trans$.
In a network,
an individual reaction ${y \to y'}$ contributes the summand $k \, x^y \, (y' - y)$ to the dynamical system for $x$,
where the reaction vector $y'-y$ captures the consumption of educts $y$ and the formation of products $y'$.
For the example reaction, $x^y=x_1 x_2$ (as stated above) and $y'-y = (-1,-1,1,0,\ldots)^\trans$.}

\blue{
Now, we can introduce a mass-action system
as a simple digraph $G=(V,E)$, a map~$y$ (assigning complexes to vertices), and edge labels $k$.
In particular,
every edge ${(i \to i') \in E}$ defines a reaction ${y(i) \to y(i')}$ with rate constant $k_{i \to i'}$.
Hence, the associated dynamical system
$\dd{x}{t} = \sum_{(i \to i') \in E} k_{i \to i'} \, x^{y(i)} \, ( y(i')-y(i) )$
involves a sum over all edges, and every summand is a product of a reaction rate and a reaction vector.
Using the Laplacian matrix $A_k$,
the right-hand-side can be decomposed as shown in Eqn.~\eqref{eq:intro}.
The matrix $Y$ collects the complexes $y(i)$ for $i \in V$,
and the vector of monomials $x^Y$ is defined via $(x^Y)_i = x^{y(i)}$.}
Altogether, the dynamical system is polynomial.
It is determined by the  
\blue{complex matrix~$Y$ (by stoichiometry)}
as well as by the \blue{Laplacian matrix~$A_k$ (by the graph),}
and chemical reaction network theory studies the interplay of these two matrices
to understand dynamics and steady states of mass-action systems,
starting from the foundational 1972 papers~\cite{HornJackson1972,Horn1972,Feinberg1972} until today.

A steady state $x > 0$ with $A_k \, x^Y=0$ is called a positive {\em complex-balanced} equilibrium (CBE),
\blue{also known as vertex-balanced steady state.
Indeed, at a CBE, the sum of all ``flows'' $k_{i \to i'} \, x^{y(i)}$ {\em from} vertex $i$/complex $y(i)$
equals the sum of all $k_{i' \to i} \, x^{y(i')}$ {\em to} the latter.}
As shown by Horn~\cite{Horn1972} and Horn \& Jackson~\cite{HornJackson1972} in 1972,
the existence of a CBE has three important consequences:
\blue{the components of the graph are strongly connected (the network is ``weakly reversible'');}
all equilibria are complex-balanced and asymptotically stable;
and there is a unique equilibrium in every \blue{dynamically invariant affine subspace (``stoichiometric compatibility class'').}
More technically, complex-balanced equilibria are given by binomial equations and have a monomial parametrization.

For \blue{symmetric digraphs (``reversible'' networks),}
{\em detailed-balanced} equilibria are given by binomial equations \blue{(by definition).
Moreover,} the polynomial dynamical system is a sum of binomials. % cf.~\cite{CraciunDickensteinShiuSturmfels2009}.
\blue{(Just note that every reversible reaction ${y \rightleftarrows y'}$ contributes the summand $(k_{y \to y'} \, x^y - k_{y' \to y} \, x^{y'})\, (y' - y)$ to the dynamical system for $x$.)
We show that this also holds for} weakly reversible networks. % (for which it is not obvious).
\blue{To this end,} we \blue{provide} a new decomposition of the graph Laplacian, 
involving an invertible \blue{{\em \irr matrix}}, based on an order on the vertices.
Further, we extend the classical result by Horn and Jackson
on the asymptotic stability of complex-balanced equilibria.
In addition to a Lyapunov function (as in classical proofs),
we consider regions in the positive orthant with given {\em monomial evaluation orders}
(and corresponding polyhedral cones in logarithmic coordinates). 
As it turns out, all dynamical systems are asymptotically stable that can be embedded in certain {\em binomial differential inclusions}.
In particular, this holds for complex-balanced mass-action systems,
and hence we also obtain a polyhedral-geometry proof of the classical result.

%Further, we give a conceptual and transparent proof 
%for the stability of complex-balanced equilibria and the non-existence of other steady states.
%In addition to a Lyapunov function (as in classical proofs),
%we use orders on the monomials, % (defining polyhedral cones),
%rather than {\em ad hoc} inequalities for the logarithm or the exponential function 
%or a cycle decomposition of the graph~\cite{HornJackson1972,Sontag2001,Gopalkrishnan2014}.
%Finally, we embed complex-balanced mass-action systems in {\em binomial} differential inclusions
%and extend the stability result.

{\bf Organization of the work.}
In Section~\ref{sec:ilm}, we provide a new decomposition of the graph Laplacian
\blue{(for labeled directed graphs \blue{with strongly connected components}),}
involving an invertible \blue{{\em \irr matrix}}, based on an order on the vertices.
%the vector of tree constants,
%and the incidence matrix of an auxiliary graph, representing an order on the vertices.
Depending on the particular order,
the \irr matrix has additional properties.

In Section~\ref{sec:mas}, we apply the graph-theoretic/algebraic results to mass-action systems.
In Subsection~\ref{sec:bin}, we demonstrate their binomial structure, 
and in~\ref{sec:mon}, we introduce {\em monomial evaluation orders} 
and corresponding geometric objects (polyhedra and polyhedral cones).
In \blue{Subsection}~\ref{sec:bdi}, we embed complex-balanced mass-action systems in {\em binomial differential inclusions}
and show that all equilibria of the latter are asymptotically stable,
\blue{and in~\ref{sec:dis}, we discuss our results.}

In Appendix~\ref{app:formulas}, we provide explicit formulas for the vector of tree constants
and the Laplacian matrix, using cycle decomposition.
In Appendix~\ref{app:aux}, we state auxiliary results used in the new decomposition of the graph Laplacian.
In Appendix~\ref{app:stab}, we give another proof of the asymptotic stability of complex-balanced equilibria
(and the non-existence of other steady states) without using differential inclusions.

{\bf Notation.}
We denote the positive real numbers by $\R_>$ and the nonnegative real numbers by $\R_\ge$.
Throughout the work, we use index notation: for a finite index set $I$,
we write $\R^I$ for the real vector space of vectors $x=(x_i)_{i \in I}$ with $x_i \in \R$,
and analogously we write $\R^I_\ge$ and $\R^I_>$.
(For $I=\{1,\ldots,n\}$, we have the standard case $\R^I=\R^n$.)
We write $x>0$ for $x \in \R^I_>$ and $x \ge 0$ for $x \in \R^I_\ge$.

For vectors $x,y \in \R^I$, we denote their scalar product by $x \cdot y \in \R$ and 
their componentwise (Hadamard) product by $x \circ y \in \R^I$.
For $x \in \R^I_>, \, y \in \R^I$, we define \blue{the (generalized) monomial} $x^y = \prod_{i \in I} (x_i)^{y_i} \in \R_>$,
\blue{and for $x \in \R^I_>, \, Y \in \R^{I \times J}$, 
we define the vector of monomials $x^Y \in \R^J_>$ via $(x^Y)_j =x^{y(j)}$,
where $y(j)$ is the column of $Y$ with index $j \in J$.}

%\blue{
%For a subset $C \subseteq \R^I$, we denote its polar cone by 
%\[ 
%C^\pol = \{ y \in \R^I \mid y \cdot x \le 0 \text{ for all } x \in C \} . 
%\]
%}
%For a matrix $A \in \R^{I \times J}$, we write $\im A \subseteq \R^I$ and $\ker A \subseteq \R^J$ for its image and kernel, respectively.

% ========= ========= ========= ========= ========= ========= ========= =========

\section{The graph Laplacian} \label{sec:ilm}

In the following, we assume that the components of a digraph are strongly connected.
%(the network is weakly reversible).
For the simplicity of the presentation,
we first consider one strongly connected component separately. % (the single linkage class case).

%In fact, several of our results also hold if the components contain exactly one terminal strong component.

\subsection*{One component}

We consider a strongly connected, simple, directed graph $G = (V,E)$ with a finite set of vertices $V = \{1, \ldots, m \}$
and a set of edges $E \subseteq V \times V$.
Further, we consider positive edge labels $k \in \R^E_>$ and the resulting labeled digraph $G_k=(V,E,k)$.
Its Laplacian matrix $A_k \in \R^{V \times V}$ is given by
\[
A_k = I_E \diag(k) I_{E,s}^\trans ,
\]
where $I_E \in \R^{V \times E}$ is the incidence matrix
and $I_{E,s} \in \R^{V \times E}$ is the ``source matrix''.
Explicitly,
\[
(A_k)_{i,j} =
\begin{cases}
k_{j \to i} , & \text{if } (j \to i) \in E , \\
-\sum_{(i \to i') \in E} k_{i \to i'} , & \text{if } i = j, \\
0 , & \text{otherwise,}
\end{cases}
\]
\[
(I_E)_{i,(j \to j')} =
\begin{cases}
-1 , & \text{if } i = j , \\
1 , & \text{if } i = j' , \\
0 , & \text{otherwise,}
\end{cases}
\]
and
\[
(I_{E,s})_{i,(j \to j')} =
\begin{cases}
1 , & \text{if } i = j , \\
0 , & \text{otherwise.}
\end{cases}
\]
This definition is used in dynamical systems.
For example,
$k$ is the vector of transition rate constants in the continuous-time, linear process $\dd{\psi}{t} = A_k \, \psi$ 
(with $\psi \in \R^V_\ge$ and $\sum_{i \in V} \psi_i=1$). 
In other fields, the Laplacian matrix is defined as $A_k^\trans$, $-A_k$, or $-A_k^\trans$.

Since $G$ is connected, $\ker I_E^\trans = \im \bar 1$, where $\bar 1 \in \R^V$ is the vector with all entries equal to one.
Further, $\ker I_{E,s}^\trans = \{0\}$.
Most importantly,
since $G$ is strongly connected,
\blue{
\begin{equation} \label{eq:Kk}
\ker A_k = \im K_k
\end{equation}
}%
with a positive vector $K_k \in \R^V_>$ (depending on the rate constants).
%\[
%A_k \, K_k = 0 .
%\]
The entries of $K_k$ (the {\em tree constants}) can be given explicitly in terms of $k$,
\[
(K_k)_i = \sum_{(V, E') \in T_i} \; \prod_{(j \to j') \in E'} k_{j \to j'} , \quad i \in V ,
\]
where $T_i$ is the set of directed spanning trees of $G$ rooted at vertex $i \in V$ 
(and directed towards the root).
\blue{For a minimal proof of Eqn.~\eqref{eq:Kk}, see \cite[Lemma~1]{Kandori1993} or Appendix~\ref{app:formulas}.}
We note that the explicit formula is not crucial for our analysis.
Finally, the tree constants $K_k$ correspond to minors of the matrix $-A_k$
which is the content of the matrix-tree theorem (for labeled, directed graphs)~\cite[Theorem~3.6]{Tutte1948}.

Clearly, the matrix
\[
- A_k \diag(K_k) \in \R^{V \times V}
\]
has positive diagonal entries and nonpositive off-diagonal entries.
Most importantly, it has zero row and column sums:
Indeed, $\bar 1^\trans A_k \diag(K_k) = 0$, and also $A_k \diag(K_k) \, \bar 1 = A_k \, K_k = 0$.
As a consequence, the matrix is diagonally dominant.

The entries of $A_k \diag(K_k)$ can be given explicitly in terms of $k$.
For a derivation of this formula and a discussion of the Birkhoff/von Neumann Theorem~\cite{Birkhoff1946,vonNeumann1953}, see Appendix~\ref{app:formulas}.
Again, we note that the explicit formula is not crucial for our analysis.

{\bf Example.}
Throughout this section, we consider the labeled directed graph $G_k=(V,E,k)$ with $V=\{1,2,3\}$ and $E=\{1 \to 2, 2 \to 1, 2 \to 3, 3 \to 1\}$,
that is,
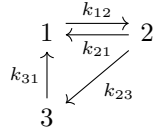
\[
\begin{tikzcd}[ampersand replacement=\&]
1 \arrow[r,"k_{12}",yshift=1ex] \& 2 \arrow[l,"k_{21}"] \arrow[ld,"k_{23}"] \\
3 \arrow[u,"k_{31}"]
\end{tikzcd}
\]
with
\[
A_k = 
\begin{pmatrix} 
-k_{12} & k_{21} & k_{31} \\
k_{12} & -k_{21} -k_{23} & 0 \\
0 & k_{23} & -k_{31}
\end{pmatrix} ,
\quad
K_k = 
\begin{pmatrix} 
k_{23} k_{31} + k_{21} k_{31} \\
k_{31} k_{12} \\
k_{12} k_{23}
\end{pmatrix} ,
\]
and
\[
A_k \diag(K_k) = k_{12} k_{23} k_{31}
\begin{pmatrix} 
-1 & 0 & 1 \\
1 & -1 & 0 \\
0 & 1 & -1
\end{pmatrix}
+ k_{12} k_{21} k_{31}
\begin{pmatrix} 
-1 & 1 & 0 \\
1 & -1 & 0 \\
0 & 0 & 0
\end{pmatrix} ,
\]
see also Appendix~\ref{app:formulas} for the cycle decomposition of $A_k \diag(K_k)$.
\hfill $\blacksquare$

Most importantly,
we introduce an {\em auxiliary} connected directed graph
$
G_\EE = (V,\EE)
$
with the same set of vertices $V$ as in $G=(V,E)$, 
but with an arbitrary set of edges $\EE \subseteq V \times V$
such that $|\EE| = |V| -1$.
That is, $G_\EE$ is a directed tree.
\blue{%
In particular, it has no cycles.
Further, $G_\EE$ need not be}
a subgraph of $G$ \blue{nor be} directed towards a root.
The corresponding incidence matrix $I_\EE \in \R^{V \times \EE}$ is given by
\[
(I_\EE)_{i,(j \to j')} =
\begin{cases}
-1 , & \text{if } i = j , \\
1 , & \text{if } i = j' , \\
0 , & \text{otherwise.}
\end{cases}
\]
Note that the definitions of the incidence matrices $I_E$ and $I_\EE$ agree formally.
(Just the sets of edges $E$ and $\EE$ differ.)
Clearly, $\ker I_\EE = \{ 0 \}$ and $\ker I_\EE^\trans = \im \bar 1$.

%On the one hand, $\ker A_k = \im K_k$; on the other hand, 
%\[
%\im K_k = \ker \left( I_\EE^\trans \diag(K_k^{-1}) \right) .
%\]
%That is, the vector $K_k \in \R^V_>$ is given by $|V|-1$ ``binomial equations''.
%Indeed, $I_\EE^\trans \diag(K_k^{-1}) K_k = I_\EE^\trans \, \bar 1 = 0$
%and $\rank \left( I_\EE^\trans \diag(K_k^{-1}) \right) = |V|-1$.

%\[
%(\ker A_k)^\perp = (\im K_k)^\perp = \ker K_k^\trans = \im \left( \diag(K_k^{-1}) I_\EE \right) .
%\]
%For the last equation, just observe $K_k^\trans \diag(K_k^{-1}) I_\EE = \bar 1^\trans I_\EE = 0$.

\begin{proposition} \label{pro1:Ac}
Let $G_k=(V,E,k)$ be a strongly connected, labeled, simple digraph
and $G_\EE = (V,\EE)$ be an auxiliary digraph. % (with $|\EE| = |V| -1$).
Then, there exists a \blue{unique} invertible matrix $\Ac \in \R^{\EE \times \EE}$, called the \blue{{\em \irr matrix} of the graph Laplacian}, such that
\[
A_k \diag(K_k) = - I_\EE \Ac I_\EE^\trans .
\]
\end{proposition}

\begin{proof}
Since $G$ is strongly connected, 
\[
\ker \left( A_k \diag(K_k) \right) = \im \bar 1 .
\]
Hence,
\[
\im \left( \diag(K_k) A_k^\trans \right) = \ker \bar 1^\trans = \im I_\EE
\]
and
\[
\diag(K_k) A_k^\trans = I_\EE B_{k,\EE}^\trans
\]
for a unique matrix $B_{k,\EE} \in \R^{V \times \EE}$,
where uniqueness follows from $\ker I_\EE = \{ 0 \}$.
For the same reason,
we have
\[
\ker A_k^\trans = \ker B_{k,\EE}^\trans
\]
and hence
\[
\im B_{k,\EE} = \im A_k .
\]
Since $G$ is strongly connected,
\[
\im A_k = \im I_E ,
\]
\blue{cf.~Lemma~\ref{lem:aux1} in Appendix~\ref{app:aux},}
%cf.~\cite[Lemma~2]{FeinbergHorn1977},
and further
\[
\im I_E =  \im I_\EE ,
\]
\blue{cf.~Lemma~\ref{lem:aux2} in Appendix~\ref{app:aux}.}
%~\cite[Proposition~5]{MuellerRegensburger2014}.
Altogether, we have
\[
\im B_{k,\EE} = \im I_\EE
\]
and hence
\[
B_{k,\EE} = - I_\EE \Ac
\]
for a unique matrix $\Ac \in \R^{\EE \times \EE}$.
(The minus sign ensures positive diagonal entries of $\Ac$ for particular auxiliary graphs; see below.)
Since $\rank(B_{k,\EE})=\rank(I_\EE)=|\EE|$, we have $\ker(\Ac) = \ker(B_{k,\EE}) = \{0\}$,
that is, $\Ac$ is invertible.
Finally, we obtain
\[
A_k \diag(K_k) = B_{k,\EE} I_\EE^\trans = - I_\EE \Ac I_\EE^\trans .
\]
\end{proof}

For an auxiliary digraph $G_\EE=(V,\EE)$, we just required $|\EE| = |V| -1$.
%that is, (the undirected version of) $G_\EE$ is a tree.
In the following two results, we assume $G_\EE$ to be either of the form $i_1 \to i_2 \to \ldots \to i_m$
(a {\em chain graph}) or of the form $i_1 \to i_m$, $i_2 \to i_m$, \ldots, $i_{m-1} \to i_m$
(a {\em star graph} with \blue{root} $i_m$).

\begin{proposition} \label{pro2:Ac}
Let $G_k=(V,E,k)$ be a strongly connected, labeled, simple digraph,
and let $G_\EE = (V,\EE)$ be an auxiliary digraph that is a chain graph.
Then $\Ac \in \R^{\EE \times \EE}$, the \blue{\irr matrix of the graph Laplacian}, is non-negative with positive diagonal.
\end{proposition}

\begin{proof}
Let $G_\EE = (V,\EE)$ be the chain graph
\[
i_1 \to i_2 \to \ldots \to i_m .
\]
It induces a natural order on the set of vertices~$V$ (and on the set of edges~$\EE$).
For $i,j \in V$, we write $i \le j$ if $i = j$ or $i \to \ldots \to j$. 
An ``inverse'' of the incidence matrix $I_\EE \in \R^{V \times \EE}$ is given by $J_\EE \in \R^{\EE \times V}$ with
\[
(J_\EE)_{(i \to i'),j} =
\begin{cases}
1 , & \text{if } j \le i , \\
0 , & \text{otherwise} .
\end{cases}
\]
Explicitly, using the order $i_1, \, i_2, \, \ldots, i_m$ on $V$,
\[
J_\EE = 
\begin{pmatrix}
1 & 0 & 0 & \cdots & 0 & 0 \\
1 & 1 & 0 & \ddots & 0 & 0 \\
\vdots & \vdots & \ddots & \ddots & \vdots & \vdots \\
1 & 1 & 1 & \ddots & 0 & 0 \\
1 & 1 & 1 & \cdots & 1 & 0
\end{pmatrix}
,
\quad
I_\EE = 
\begin{pmatrix}
-1 & 0 & \cdots & 0 & 0 \\
1 & -1 & \ddots & \vdots & \vdots \\
0 & 1 & \ddots & 0 & \vdots \\
\vdots & 0 & \ddots & -1 & 0 \\
\vdots & \vdots & \ddots& 1 & -1 \\
0  & 0 & \cdots & 0 & 1
\end{pmatrix}
,
\]
and indeed, $J_\EE I_\EE = - \mathrm{I}$,
where $\mathrm{I} \in \R^{\EE \times \EE}$ is the identity matrix.
\blue{That is, $-J_\EE$ is a generalized left-inverse of $I_\EE$.}
Hence, by Proposition~\ref{pro1:Ac},
\[
\Ac = - J_\EE A_k \diag(K_k) J_\EE^\trans .
\]
For an arbitrary matrix $A \in \R^{V \times V}$, 
\begin{equation*} \label{eq:sumC} \tag{$\sigma$}
(J_\EE A \, J_\EE^\trans)_{i \to i', j \to j'} = \sum_{\bar i \colon \bar i \le i} \; \sum_{\bar j \colon \bar j \le j} A_{\bar i,\bar j} .
\end{equation*}
Explicitly,
\eqref{eq:sumC} is the sum of all entries in the upper left $i \times j$ block of~$A$.
Now, recall that the matrix $A = -A_k \diag(K_k)$ 
has positive diagonal entries and nonpositive off-diagonal as well as zero row and column sums.
Hence, the sum \eqref{eq:sumC} is nonnegative.
Finally, recall that the underlying graph $G$ is strongly connected.
If $i \to i'$ equals $j \to j'$, then the sum \eqref{eq:sumC} is positive, 
since the corresponding subgraph with vertices $\{i_1, i_2, \ldots, i\}$ has incoming and outgoing edges.
\end{proof}

{\bf Example (continued).} 
In the labeled digraph $G_k=(V,E,k)$ introduced above, there are 3 vertices and hence 6 possible {\em chain graphs}.
For example, for $\EE = \{ 1 \to 2, 2 \to 3 \}$ (contained in $E$), we find
\[
\Ac = - J_\EE A_k \diag(K_k) J_\EE^\trans = 
k_{12} k_{23} k_{31}
\blue{
\begin{pmatrix} 
1 & 1 \\
0 & 1
\end{pmatrix}}
+ k_{12} k_{21} k_{31}
\begin{pmatrix} 
1 & 0 \\
0 & 0
\end{pmatrix} ,
\]
whereas for $\EE = \{ 1 \to 3, 3 \to 2\}$ (both edges not contained in $E$), we find
\[
\Ac = k_{12} k_{23} k_{31}
\blue{
\begin{pmatrix} 
1 & 0 \\
1 & 1
\end{pmatrix}}
+ k_{12} k_{21} k_{31}
\begin{pmatrix} 
1 & 1 \\
1 & 1
\end{pmatrix} .
\]

\begin{proposition} \label{pro3:Ac}
Let $G_k=(V,E,k)$ be a strongly connected, labeled, simple digraph,
and let $G_\EE = (V,\EE)$ be an auxiliary digraph that is a star graph.
Then $\Ac \in \R^{\EE \times \EE}$, the \blue{\irr matrix of the graph Laplacian}, is
\blue{(row and column)} diagonally dominant with positive diagonal and non-positive off-diagonal entries.

Explicitly, let $G_\EE = (V,\EE)$ have root $i_m \in V$.
Then $\Ac \in \R^{\EE \times \EE}$ equals $-A_k \diag(K_k) \in \R^{V \times V}$ with row $i_m$ and column $i_m$ removed
\blue{and} edges ${(i \to i_m) \in \EE}$ \blue{identified} with vertices $i \in V \setminus \{i_m\}$.

\end{proposition}

\begin{proof}
Let $G_\EE = (V,\EE)$ be the star graph
\[
i_1 \to i_m, \, i_2 \to i_m, \, \ldots, \, i_{m-1} \to i_m .
\]
An ``inverse'' of the incidence matrix $I_\EE \in \R^{V \times \EE}$ is given by $J_\EE \in \R^{\EE \times V}$ with
\[
(J_\EE)_{(i \to i'),j} =
\begin{cases}
0 , & \text{if } j=i , \\
1 , & \text{otherwise} .
\end{cases}
\]
Explicitly, using the order $i_1, \, i_2, \, \ldots, i_m$ on $V$,
\[
J_\EE = 
\begin{pmatrix}
0 & 1 & \cdots & \cdots & 1 & 1 \\
1 & 0 & 1 & \cdots & 1 & 1 \\
\vdots & \ddots & \ddots & \ddots & \vdots & \vdots \\
1 & \cdots & 1 & 0 & 1 & 1 \\
1 & \cdots & \cdots & 1 & 0 & 1 
\end{pmatrix}
,
\quad
I_\EE = 
\begin{pmatrix}
-1 & 0 & \cdots & 0 & 0 \\
0 & -1 & \ddots & \vdots & \vdots \\
\vdots & 0 & \ddots & 0 & \vdots \\
\vdots & \vdots & \ddots& 1 & 0 \\
0  & 0 & \cdots & 0 & -1 \\
1 & 1 & \cdots & 1 & 1
\end{pmatrix}
,
\]
and indeed, $J_\EE I_\EE = \mathrm{I}^{\EE \times \EE}$.
\blue{That is, $J_\EE$ is a generalized left-inverse of $I_\EE$.}
Hence, by Proposition~\ref{pro1:Ac},
\[
\Ac = - J_\EE A_k \diag(K_k) J_\EE^\trans .
\]
For an arbitrary matrix $A \in \R^{V \times V}$, 
\begin{equation*} \label{eq:sumS} \tag{$\sigma^\star$}
(J_\EE A \, J_\EE^\trans)_{i \to i_m,j \to i_m} = \sum_{\bar i \colon \bar i \neq i} \; \sum_{\bar j \colon \bar j \neq j} A_{\bar i,\bar j} .
\end{equation*}
That is,
\eqref{eq:sumS} is the sum of all entries of $A$ except the entries in row~$i$ and column~$j$.
Now, recall that the matrix $A = -A_k \diag(K_k)$ 
has %positive diagonal entries and nonpositive off-diagonal as well as 
\blue{zero row and column sums.
Hence, \eqref{eq:sumS} equals the sum of all entries (which is zero)
{\em minus} the sums of all entries in row $i$ and column $j$ (which are zero)
{\em plus} the common entry of row $i$ and column $j$.
That is,
\[
(\Ac)_{i \to i_m,j \to i_m} =
- (J_\EE A_k \diag(K_k) \, J_\EE^\trans)_{i \to i_m,j \to i_m} = - (A_k)_{i,j} (K_k)_{j} .
\]}%
%and
%\[
%(\Ac)_{i \to i_m,j \to i_m} = - (A_k)_{i,j} (K_k)_{j} .
%\]
As claimed, $\Ac$ equals $\blue{A =} -A_k \diag(K_k)$ with row $i_m$ and column $i_m$ removed.
\blue{Like $A$,} it has positive diagonal entries and nonpositive off-diagonal entries
and is \blue{(row and column)} diagonally dominant.
(However, not all row and column sums are zero.)
\end{proof}

{\bf Example (continued).} 
In the labeled digraph $G_k=(V,E,k)$ introduced above, there are 3 vertices and hence 3 possible {\em star graphs}.
For example, for $\EE = \{ 2 \to 1 , 3 \to 1 \}$ (contained in $E$), we find
\[
\Ac = - J_\EE A_k \diag(K_k) J_\EE^\trans = 
k_{12} k_{23} k_{31}
\begin{pmatrix} 
1 & 0 \\
-1 & 1
\end{pmatrix}
+ k_{12} k_{21} k_{31}
\begin{pmatrix} 
1 & 0 \\
0 & 0
\end{pmatrix} ,
\]
whereas for $\EE = \{ 1 \to 3, 2 \to 3 \}$ (first edge not contained in $E$), we find
\[
\Ac = k_{12} k_{23} k_{31}
\begin{pmatrix} 
1 & 0 \\
-1 & 1
\end{pmatrix}
+ k_{12} k_{21} k_{31}
\begin{pmatrix} 
1 & -1 \\
-1 & 1
\end{pmatrix} .
\]

\blue{
{\bf Remark.}
In applications to mass-action systems in Section~\ref{sec:mas}, we use chain graphs (rather than star graphs).}

% ========= ========= ========= ========= ========= ========= ========= =========

\subsection*{Several components}

In general,
we consider a labeled, simple digraph $G_k=(V,E,k)$ with $\ell$ strongly connected components $\Gkl = (\Vl, \El, \kl)$,
$\la=1,\ldots,\ell$,
finite sets of vertices $\Vl$, sets of edges $\El \subseteq \Vl \times \Vl$, and positive edge labels $\kl \in \R^{\El}_>$.
The corresponding Laplacian matrix $A_k \in \R^{V \times V}$ is block-diagonal with blocks
$\Akl \in \R^{\Vl \times \Vl}$,
and the vector of tree constants $K_k \in \R^V_>$ has blocks $\Kkl \in \R^{\Vl}_>$.
\blue{
Explicitly,
\[
A_k = \begin{pmatrix} A_k^1 & & 0 \\ & \ddots & \\ 0 & & A_k^\ell \end{pmatrix} \in \R^{V \times V}
\quad\text{and}\quad
K_k = \begin{pmatrix} K_k^1 \\ \vdots \\ K_k^\ell \end{pmatrix} \in \R^V_> .
\]
}

Accordingly,
an auxiliary digraph $G_\EE = (V,\EE)$ has $\ell$ connected components $\GEEl = (\Vl, \EEl)$ with $\EEl \subseteq \Vl \times \Vl$ and $|\EEl| = |\Vl| -1$. 
The corresponding incidence matrix $I_\EE \in \R^{V \times \EE}$ is block-diagonal with blocks $\IEEl \in \R^{\Vl \times \EEl}$.
\blue{We say that $G_\EE$ is a chain graph, if each component of $G_\EE$ is a chain graph,
and analogously for a star graph.}

Propositions~\ref{pro1:Ac}, \ref{pro2:Ac}, and \ref{pro3:Ac} imply the main result of this section.

\begin{theorem} \label{thm:Ac}
Let $G_k=(V,E,k)$ be a labeled, simple digraph with strongly connected components,
and let $G_\EE = (V,\EE)$ be an auxiliary digraph. % (with $|\EE| = |V| -1$).
Then, there exists an invertible, block-diagonal matrix $\Ac \in \R^{\EE \times \EE}$, 
called the \blue{{\em \irr matrix} of the graph Laplacian}, such that
\[
A_k \diag(K_k) = - I_\EE \Ac I_\EE^\trans .
\]
If %\blue{each component of} 
$G_\EE$ is a chain graph,
then $\Ac$ is non-negative with positive diagonal.
If %\blue{each component of}
$G_\EE$ is a star graph,
then $\Ac$ is diagonally dominant with positive diagonal and non-positive off-diagonal entries.
\end{theorem}

\blue{
Explicitly,
\[
I_\EE = \begin{pmatrix} I_\EE^1 & & 0 \\ & \ddots & \\ 0 & & I_\EE^\ell \end{pmatrix} \in \R^{V \times \EE}
\quad\text{and}\quad
\Ac = \begin{pmatrix} \Ac^1 & & 0 \\ & \ddots & \\ 0 & & \Ac^\ell \end{pmatrix} \in \R^{\EE \times \EE} .
\]
Note that $|\EEl| = |\Vl|-1$, $\lambda=1,\ldots,\ell$, and hence $|\EE|=|V|-\ell$.
That is, an auxiliary graph has $|\EE|=|V|-\ell$ edges,
and a \blue{\irr matrix} has $|\EE|=|V|-\ell$ rows and columns.
}

% ========= ========= ========= ========= ========= ========= ========= =========

\section{Mass-action systems} \label{sec:mas}

We apply the graph-theoretic/algebraic results from the previous section to mass-action systems.
We start with a brief summary of fundamental concepts and results.

A {\em chemical reaction network} $(G,y)$
is given by a simple directed graph $G = (V,E)$ with a finite set of vertices $V = \{1, \ldots, m \}$
and a set of edges {\em (reactions)} $E \subseteq V \times V$
together with an \blue{injective} map $y \colon V \to \R^n_\ge$ (a matrix $Y \in \R^{n \times V}_\ge$),
assigning to every vertex $i \in V$ a {\em complex} $y(i) \in \R^n_\ge$.
(The digraph $G$ is ``embedded'' in $\R^n_\ge$.)
If the components of $G$ (the {\em linkage classes}) are strongly connected, then the network is called {\em weakly reversible}.

A {\em mass-action system} $(G_k,y)$
is a chemical reaction network $(G,y)$
where every edge $(i \to i') \in E$ is labeled with a {\em rate constant} $k_{i \to i'} > 0$,
yielding the labeled, simple digraph $G_k = (V,E,k)$ with $k \in \R^E_>$.
(If the network is weakly reversible, then also the mass-action system is called weakly reversible.)

The resulting dynamical system for $x \in \R^n_\ge$ (the {\em concentrations} of $n$ molecular species) is given by
\[
\dd{x}{t} = f_k(x) = \sum_{(i \to i') \in E} k_{i \to i'} \, x^{y(i)} \left(y(i') - y(i)\right) .
\]
The right-hand side of the ODE can be decomposed as
\[
f_k(x) = Y I_E \diag(k) I_{E,s}^\trans \, x^Y = Y A_k \, x^Y 
\]
where $I_E \in \R^{V \times E}$ is the incidence matrix,
$I_{E,s} \in \R^{V \times E}$ is the ``source matrix'',
and
\[
A_k = I_E \diag(k) I_{E,s}^\trans\in \R^{V \times V}
\]
is the resulting Laplacian matrix of the labeled, simple digraph $G_k$.
In the following, we consider the dynamical system in the form
\begin{equation} \label{dynsysAk}
\dd{x}{t} = f_k(x) = Y A_k \, x^Y .
\end{equation}

The {\em stoichiometric subspace} is given by
$
S = \im (Y I_E) .
$
Clearly, $\dd{x}{t} = f_k(x) \in S$, and hence $x(t) \in x(0)+S$.
For $x' \in \R^n_>$, the forward invariant set $(x'+S) \cap \R^n_\ge$ is called a positive {\em stoichiometric (compatibility) class}.

If an equilibrium $x \in \R^n_>$ of the ODE fulfills
\begin{equation}
A_k \, x^Y = 0 ,
\end{equation}
then it is a positive {\em complex-balanced} equilibrium (CBE), 
\blue{also known as vertex-balanced steady state.}

\blue{{\bf Remark.}
In the linear setting, the Laplacian matrix captures state transitions on a graph.
Let $\psi = x^Y$ be the state variable, given by the vector of monomials.
If $A_k \, \psi = 0$, 
then transitions are balanced (at every vertex of the graph), and $x$ is a CBE.
If $Y A_k \, x^Y = 0$ (but not $A_k \, x^Y = 0$), then $x$ is a general equilibrium.
}

As shown by Horn~\cite{Horn1972} and Horn \& Jackson~\cite{HornJackson1972},
if there exists a positive CBE (in some stoichiometric class),
then 
\begin{enumerate}
\item
the mass-action system is weakly reversible~\cite[Theorem~3C]{Horn1972},
\item
the equilibrium is asymptotically stable, and all equilibria are complex-balanced~\cite[Theorem~6A]{HornJackson1972},
and, 
\item
there exists a unique positive (necessarily complex-balanced) equilibrium in every stoichiometric class~\cite[Lemma~4B]{HornJackson1972}.
\end{enumerate}

In the following remarks, we elaborate on results 1, 2, and 3.

{\bf Remark} (result 1).
Let $G$ be weakly reversible
and $G_\EE = (V,\EE)$ be some auxiliary digraph.
By Theorem~\ref{thm:Ac}, $A_k = - I_\EE \Ac I_\EE^\trans \diag(K_k^{-1})$.
Further, $\ker(I_\EE) = \ker(\Ac) = \{0\}$.
Hence, a positive CBE $x \in \R^n_>$ is given by
\[
I_\EE^\trans \diag(K_k^{-1}) \, x^Y = 0 ,
\]
that is, by the binomial equations
\begin{equation} \label{cbe}
\frac{x^{y(i')}}{(K_k)_{i'}} - \frac{x^{y(i)}}{(K_k)_i} = 0 \quad \text{for } (i \to i') \in \EE .
\end{equation}
Given a particular positive CBE $x^* \in \R^n_>$,
Eqn.~\eqref{cbe} is equivalent to 
\[
\left(\frac{x}{x^*}\right)^{y(i')} = \left(\frac{x}{x^*}\right)^{y(i)} \quad \text{for } (i \to i') \in \EE
\]
and further to $(y(i')-y(i))^\trans \ln (x/x^*)=0$ for $(i \to i') \in \EE =0$, that is, to ${(Y I_\EE)^\trans \ln (x/x^*) = 0}$.
Since $\ker (Y I_\EE)^\trans = (\im Y I_\EE)^\perp = (\im Y I_E)^\perp = S^\perp$,
the set of all positive CBEs is given by the monomial parametrization $x= x^* \circ \e^{S^\perp}$.

{\bf Remark} (result 2).
In Section~\ref{sec:bdi}, we extend the classical stability result.
As it turns out, it holds not only for complex-balanced equilibria of mass-action systems,
but for all equilibria of binomial differential inclusions.

In Appendix~\ref{app:stab}, we give another proof for the asymptotic stability of complex-balanced equilibria
(and the non-existence of other steady states) without using differential inclusions.

{\bf Remark} (result 3).
Technically, result 3 states that $|(x^* \circ \e^{S^\perp}) \cap (x'+S)| = 1$,
for all $x^*, x' \in \R^n_>$.
An equivalent result appears in toric geometry \cite{Fulton1993}, where it is related to moment maps,
and in statistics \cite{PachterSturmfels2005}, where it is related to log-linear models
and called Birch's theorem after~\cite{Birch1963}.
For generalizations, see \cite{MuellerRegensburger2012,MuellerRegensburger2014,MuellerHofbauerRegensburger2019,CraciunMuellerPanteaYu2019}
and \cite{GopalkrishnanMillerShiu2014}.

%In a second paper~\cite{Horn1972} of 1972,
%Horn showed that, under the assumptions of Theorem~\ref{thm:HJ},
%there exists a positive (necessarily complex-balanced) equilibrium in every stoichiometric class.
%In mathematical terms, Theorem~\ref{thm:HJ} is based on an entropy-like {\em Lyapunov function},
%whereas the second paper proves a result also known as {\em Birch's theorem}.

% ========= ========= ========= ========= ========= ========= ========= =========

\subsection{Binomial structure} \label{sec:bin}

Given that the network is weakly reversible (the components of the graph $G$ are strongly connected),
our main graph-theoretic/algebraic result, Theorem~\ref{thm:Ac}, implies that 
the dynamical system~\eqref{dynsysAk} for the mass-action system $(G_k,y)$ can be decomposed as
\begin{equation} \label{dynsysAc}
\dd{x}{t} = f_{k,\EE}(x) = - Y I_\EE \Ac I_\EE^\trans \diag(K_k^{-1}) \, x^Y ,
\end{equation}
\blue{where} $G_\EE = (V,\EE)$ is some auxiliary digraph.

Again, we have a closer look at the term $I_\EE^\trans \diag(K_k^{-1}) \, x^Y \in \R^\EE$.
Indeed,
\[
\left( I_\EE^\trans \diag(K_k^{-1}) \, x^Y \right)_{i \to i'} =
\frac{x^{y(i')}}{(K_k)_{i'}} - \frac{x^{y(i)}}{(K_k)_i} \quad \text{for } (i \to i') \in \EE .
\]
That is, 
the right-hand side of the dynamical system is a sum of binomials.
This is obvious for symmetric digraphs (reversible networks); \blue{cf.~\cite[Eqn.~(14)]{CraciunDickensteinShiuSturmfels2009}.}
By Theorem~\ref{thm:Ac}, it also holds for digraphs with strongly connected components (weakly reversible networks).

In particular, for a complex-balanced equilibrium, not just the right-hand side of \eqref{dynsysAc} is zero,
but every individual binomial is zero.
In this sense,
the ODE \eqref{dynsysAc} does not only have binomial {\em steady states} (positive complex-balanced equilibria,
given by binomial equations),
but truly is a binomial {\em dynamical system}.
%In the language of algebraic geometry,
%the notion ``toric'' is used for ``binomial''~\cite{CraciunDickensteinShiuSturmfels2009,PerezMillanDickensteinShiuConradi2012}.

%As the right hand-side of the dynamical system~\eqref{dynsysAk},
%also the right hand-side of the dynamical system~\eqref{dynsysAc} 
%is a sum of $\ell$ terms, corresponding to the $\ell$ connected components.
%We write $B = Y I_\EE \in \R^{n \times \EE}$ and $b = I_\EE^\trans \diag(K_k^{-1}) \, x^Y \in \R^\EE$.
%Then,
%\begin{equation} \label{eq:block}
%Y I_\EE \Ac I_\EE^\trans \diag(K_k^{-1}) \, x^Y = B \, \Ac b = \sum_{\la = 1,\ldots,\ell} B^{\la} \Ac^{\la} \, b^{\la} ,
%\end{equation}
%where 
%\begin{align*}
%\Ac^{\la} &\in \R^{\EEl \times \EEl}, \\
%B^{\la} &= Y^{\la} \IEEl \in \R^{n \times \EEl} \text{ with } Y^{\la} \in \R^{n \times \Vl} , \text{ and} \\
%b^{\la} &= (\IEEl)^\trans \diag((\Kkl)^{-1}) \, (x^Y)^{\la} \in \R^\EEl \text{ with } (x^Y)^{\la} \in \R^\Vl_>
%\end{align*}
%are the corresponding blocks of $\Ac$, $B$, and $b$.

% ========= ========= ========= ========= ========= ========= ========= =========

\subsection{Monomial evaluation orders and
corresponding polyhedra/polyhedral cones} \label{sec:mon}

Let $(G_k,y)$ be a %weakly reversible 
mass-action system 
based on the labeled, simple digraph $G_k=(V,E,k)$ and the map $y$ (the matrix $Y$).

For \blue{fixed} $x \in \R^n_>$,
the values of the monomials $x^{y(i)}$ with $i \in V$ are ordered \blue{(using the order on $\R$).}
For simplicity, we first consider a connected graph $G=(V,E)$.
Obviously, 
the {\em total} order
\[
x^{y(i_1)} \le x^{y(i_2)} \le \ldots \le x^{y(i_m)} 
\]
can be represented by a chain graph,
\[
i_1 \to i_2 \to \ldots \to i_m.
\]
If the order is non-strict \blue{(if some monomials have the same value),} then the representation is not unique.
Analogously,
the {\em partial} order
\[
x^{y(i_1)} \le x^{y(i_m)}, \, x^{y(i_2)} \le x^{y(i_m)}, \,  \ldots, \, x^{y(i_{m-1})} \le x^{y(i_m)}
\]
can be represented by a star graph,
\[
i_1 \to i_m, \, i_2 \to i_m, \, \ldots, \, i_{m-1} \to i_m .
\]
In \blue{general}, every auxiliary graph  $G_\EE=(V,\EE)$ represents a partial order on the vertices of $G$
and hence on the \blue{values of the} monomials.

In the following, %for our results, 
we will consider monomials with coefficients:
\begin{itemize}
\item
$\frac{x^{y(i)}}{(K_k)_i}$,
for weakly reversible networks with tree constants $K_k \in \R^V_>$, and
%leading to non-central hyperplane arrangements, and
\item
$(\frac{x}{x^*})^{y(i)}$,
for given positive CBE $x^* \in \R^n_>$.
%leading to central hyperplane arrangements.
\end{itemize}

\subsubsection*{Weak reversibility}

Let $(G_k,y)$ be a weakly reversible mass-action system,
and \blue{fix} $x \in \R^n_>$.

We call an % total (strict or non-strict) 
order on the entries of $\frac{x^Y}{K_k} \in \R^V_>$ 
\blue{that is total within connected components,
but does not relate entries in different components,}
%and hence on the vertices $V$
a {\em monomial evaluation order}
(since the notion {\em monomial order\blue{(-ing)}} has a different meaning in algebra).
We represent the order by a chain graph $G_{\EE} = (V,\EE)$
and often just by the set of edges $\EE$.
Explicitly,
%for $\Vl = \{ i_1, \ldots, i_{m\la} \}$ and $x^{y(i_1)} \le \ldots \le x^{y(i_{m\la})}$,
%$G_\EE^{\la} = (\Vl,\EEl)$ equals $i_1 \to \ldots \to i_{m\la}$.
$(i \to i') \in \EE$ 
implies $\frac{x^{y(i)}}{(K_k)_i} \le \frac{x^{y(i')}}{(K_k)_{i'}}$.
Thereby, the vertices $i,i' \in V$ are necessarily in the same %strongly connected 
component.
% that is, $(y(i')^\trans - y(i')^\trans) \ln x \ge \ln \frac{(K_k)_{i'}}{(K_k)_i}$.
If the order is non-strict, then $\EE$ is not unique.

Analogously, the maximal entries of $\frac{x^Y}{K_k} \in \R^V_>$ \blue{within connected components} are greater or equal than all other entries \blue{in the respective components.}
We represent this %partial 
order by a star graph $G_{\EE} = (V,\EE)$.
%Explicitly,
%for $i, i' \in \Vl$, $(i \to i') \in \EE^{\la}_\star$ implies $\frac{x^{y(i)}}{(K_k)_{i}} \le \frac{x^{y(i')}}{(K_k)_{i'}}$.
If there is more than one maximal entry \blue{within a component,} 
then $\EE$ is not unique.

Conversely,
\blue{fix an auxiliary graph}
$G_\EE=(V,\EE)$,
for example, a chain graph or a star graph.
The subset of $\R^n_>$ with monomial evaluation order represented by $\EE$
is given by
\begin{equation} \label{st1}
\begin{aligned} 
\ST 
&= \left\{ x \in \R^n_> \mid \frac{x^{y(i')}}{(K_k)_{i'}} - \frac{x^{y(i)}}{(K_k)_i} \ge 0 \text{ for } (i \to i') \in \EE \right\} \\
&= \left\{ x \in \R^n_> \mid I_\EE^\trans \diag(K_k^{-1}) \, x^Y \ge 0 \right\} .
\end{aligned}
\end{equation}
By the monotonicity of the logarithm,
\begin{align*}
\ST 
&= \left\{ x \in \R^n_> \mid (y(i')-y(i))^\trans \ln x \ge \ln \frac{(K_k)_{i'}}{(K_k)_i} \text{ for } (i \to i') \in \EE \right\} \\
&= \left\{ x \in \R^n_> \mid (Y I_\EE)^\trans \ln x \ge I_\EE^\trans \ln K_k \right\} .
\end{align*}
Hence,
\[
x \in \ST
\quad \Leftrightarrow \quad
\ln x \in \PE 
\]
with the polyhedron
\begin{equation} \label{pe}
\begin{aligned} 
\PE
&= \left\{ z \in \R^n \mid (Y I_\EE)^\trans z \ge I_\EE^\trans \ln K_k \right\} .
%&= \left\{ z \in \R^n \mid I_\EE^\trans \diag(K_k^{-1}) \e^{Y^\trans z} \ge 0 \right\} .
%&= \left\{ z \in \R^n \mid (y(i')-y(i))^\trans z \ge \ln (K_k)_{i'} - \ln (K_k)_i \text{ for } (i \to i') \in \EE \right\} \\
\end{aligned}
\end{equation}

\subsubsection*{Complex balancing}

If there exists a positive CBE $x^\ast \in \R^n_>$,
then the polyhedra become polyhedral cones.

\blue{Fix an auxiliary graph} $G_\EE=(V,\EE)$.
Using complex balancing~\eqref{cbe} \blue{for $x^*$}, 
the subset~\eqref{st1} can be written as
\begin{align*}
\ST 
&= \left\{ x \in \R^n_> \mid \left(\frac{x}{x^*}\right)^{y(i')} - \left(\frac{x}{x^*}\right)^{y(i)} \ge 0 \text{ for } (i \to i') \in \EE \right\} \\
&= \left\{ x \in \R^n_> \mid I_\EE^\trans \left(\frac{x}{x^*}\right)^Y \ge 0 \right\} .
\end{align*}
By the monotonicity of the logarithm,
\begin{align*} %\label{st2}
\ST 
&= \left\{ x \in \R^n_> \mid (y(i')-y(i))^\trans \ln \frac{x}{x^*} \ge 0 \text{ for } (i \to i') \in \EE \right\} \\
&= \left\{ x \in \R^n_> \mid (Y I_\EE)^\trans \ln \frac{x}{x^*} \ge 0 \right\} . \nonumber
\end{align*}
Hence,
\[
x \in \ST
\quad \Leftrightarrow \quad
\ln \frac{x}{x^*} \in \CE 
\]
with the polyhedral cone
\begin{align} \label{ce}
\CE 
%&= \left\{ z \in \R^n \mid (y(i')-y(i))^\trans z \ge 0 \text{ for } (i \to i') \in \EE \right\} \\
&= \left\{ z \in \R^n \mid (Y I_\EE)^\trans z \ge 0 \right\} ,
\end{align}
which does not depend on $k$.
(Of course, $x^*$ depends on $k$.)
The lineality space of $\CE$ does not even depend on~$\EE$,
\[
\lin \CE = \ker \, (Y I_\EE)^\trans = (\im Y I_\EE)^\perp = (\im Y I_E)^\perp = S^\perp .
\]
%The arrangement of central hyperplanes
%\[
%h_{i \to i'} = \left\{ z \in \R^n \mid (y(i')-y(i))^\trans z = 0 \right\}, \quad (i \to i') \in \Omega .
%\]
%decomposes $\R^n$ into (open) polyhedral cones.

\blue{
Obviously, $S^\perp = \lin \CE \subseteq \CE$.
For fixed $\EE$, there are two possibilities:}
\begin{itemize}
\item
$\CE = S^\perp$.
Then, all defining \blue{(non-strict)} inequalities of $\CE$ \blue{(and $\ST$)} are fulfilled with equality,
and $\ST = x^* \circ \e^{S^\perp}$ equals the set of complex-balanced equilibria.
% That is, there is no $x \in \ST$ with strict order
\item
\blue{$\CE \supset S^\perp$. Then $\CE$ and $\ST$ are full-dimensional,}
and the monomial evaluation order is
strict in the interior of $\ST$ and non-strict on the boundary \blue{(where some monomials have the same value).}
\end{itemize}

In the following study of complex-balanced mass-action systems (and their extension to binomial differential inclusions),
we use chain graphs $G_\EE$, representing monomial evaluation orders.
In this setting, a full-dimensional subset $\ST$ is called a {\em stratum},
cf.~\cite{SiegelJohnston2011}.
This term has also been used for partial orders related to the original graph, rather than to an auxiliary graph,
cf.~\cite{CraciunDickensteinShiuSturmfels2009}.

\blue{\bf Remark.}
As stated above, for every $x \in \R^n_>$,
there is a (non-unique) $\EE$ such that $x \in \ST$.
In particular, $\R^n_>$ is a union of strata which intersect only on their boundaries.
\blue{Correspondingly, $\R^n$ is a union of polyhedral cones $\CE$.
Indeed, by the monotonicity of the logarithm, an order on the entries of ${(\frac{x}{x^*})^Y \in \R^V_>}$ (within components)
is equivalent to an order on the entries of $Y^\trans z \in \R^V$ with $z = \ln \frac{x}{x^*}$,
and the set of pairs of vertices within components,
\[
\Omega = \left\{ i \to i' \mid i,i' \in V^\la, \, \la=1,\ldots,\ell \right\} ,
\]
induces an arrangement of central hyperplanes,
\[
h_{i \to i'} = \{ z \in \R^n \mid ( y(i') - y(i))^T z = 0 \}, \quad (i \to i') \in \Omega .
\]
The central hyperplane arrangement decomposes $\R^n$ into open polyhedral cones called {\em faces};
full dimensional faces are called {\em cells}.
In our terminology, a cell is the interior of a polyhedral cone $\CE$
and hence corresponds to the interior of a stratum $\ST$.
}

%\clearpage

{\bf Example.} Let $(G_k,y)$ be a mass-action system given by 
a strongly connected graph $G=(V,E)$ with $V=\{1,2,3\}$ (and arbitrary $E$)
and $y(1) = {2 \choose 1}$, $y(2) = {0 \choose 2}$, $y(3) = {1 \choose 0}$.
For simplicity, assume $x^* = {1 \choose 1}$.
The corresponding monomials are $(x/x^*)^{y(1)} = x^{y(1)} = x_1^2 x_2$, $x^{y(2)} = x_2^2$, and $x^{y(3)} = x_1$.

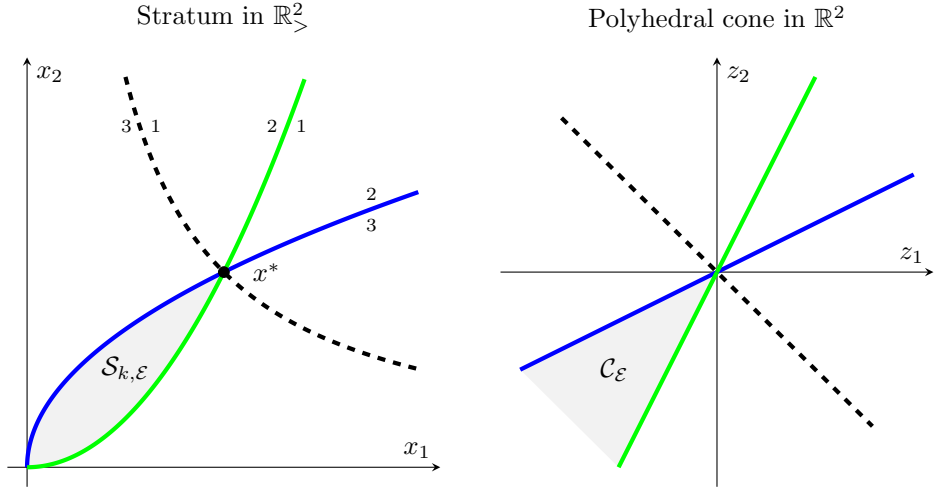
\begin{tikzpicture}
\begin{axis}[width=0.6\textwidth,height=0.6\textwidth,
    axis lines=middle,xtick=\empty,ytick=\empty,
    xmin=-.1,xmax=2.1,ymin=-.1,ymax=2.1,samples=100,
    xlabel={$x_1$},ylabel={$x_2$},title={Stratum in $\R^2_>$}]
    \addplot[black, ultra thick, dashed, domain=.5:2] (x,1/x);
    \addplot[name path = A, blue, ultra thick, domain=0:1] (x*x,x);
    \addplot[blue, ultra thick, domain=1:1.41] (x*x,x);
    \addplot[name path = B, green, ultra thick, domain=0:1] (x,x*x);
    \addplot[green, ultra thick, domain=1:1.41] (x,x*x);
    \addplot [gray!20, opacity=0.5] fill between [of = A and B, soft clip={}];
    \addplot[mark=*] coordinates {(1,1)}; \node [right] at (1.1,1) {$x^*$};
    \node [right] at (.33,.5) {$\ST$};
    %\addplot[mark=*] coordinates {(0.57,1.75)}; \addplot[mark=*] coordinates {(1.32,1.75)}; \addplot[mark=*] coordinates {(1.75,1.32)};
    \node [left] at (0.58,1.75) {$\scriptstyle 3$};\node [right] at (0.58,1.75) {$\scriptstyle 1$};
    \node [left] at (1.32,1.75) {$\scriptstyle 2$};\node [right] at (1.32,1.75) {$\scriptstyle 1$};
    \node [above] at (1.75,1.32) {$\scriptstyle 2$};\node [below] at (1.75,1.32) {$\scriptstyle 3$};
\end{axis}
\end{tikzpicture}
\qquad
\begin{tikzpicture}
\begin{axis}[width=0.6\textwidth,height=0.6\textwidth,
    axis lines=middle,xtick=\empty,ytick=\empty,
    xmin=-1.1,xmax=1.1,ymin=-1.1,ymax=1.1,samples=10,
    xlabel={$z_1$},ylabel={$z_2$},title={Polyhedral cone in $\R^2$}] 
    \addplot[black, ultra thick, dashed, domain=-.79:.79]{-x};
    \addplot[name path = A, blue, ultra thick, domain=-1:0]{x/2};
    \addplot[blue, ultra thick, domain=0:1]{x/2};
    \addplot[green, ultra thick, domain=0:1](x/2,x);
    \addplot[name path = B, green, ultra thick, domain=-1:0](x/2,x);
    \addplot[gray!20, opacity=0.5] fill between [of = A and B, soft clip={}];
    \node [right] at (-.64,-.5) {$\CE$};
\end{axis}
\end{tikzpicture}

The positive orthant is a union of strata corresponding to monomial evaluation orders.
In particular, consider the stratum given by the order $x^{y(1)} \le x^{y(2)} \le x^{y(3)}$,
that is, $\ST$ with $\EE = \{1 \to 2, 2 \to 3\}$, bounded by the green and blue lines.
The green line specifies $x^{y(1)} = x^{y(2)}$; above it, $x^{y(2)} > x^{y(1)}$, as indicated by the corresponding vertices 2 and 1.
The blue line specifies $x^{y(2)} = x^{y(3)}$; below it, $x^{y(3)} > x^{y(2)}$.
(The \blue{dashed black} line specifies $x^{y(1)} = x^{y(3)}$, which does not bound the particular stratum.)
In the interior of $\ST$, the order is strict.
In logarithmic coordinates $z = \ln (x/x^*)$, the stratum corresponds to the polyhedral cone~$\CE$.

\begin{tikzpicture}
\begin{axis}[width=0.6\textwidth,height=0.6\textwidth,
    axis lines=middle,xtick=\empty,ytick=\empty,
    xmin=-.1,xmax=2.1,ymin=-.1,ymax=2.1,samples=100,
    xlabel={$x_1$},ylabel={$x_2$},title={\small No stratum for $\EE=\{1\to2\to3,\,4\to5\}$}]
    \addplot[red, ultra thick, domain=.5:2] (x,1/x);
    \addplot[name path = A, blue, ultra thick, domain=0:1] (x*x,x);
    \addplot[blue, ultra thick, domain=1:1.41] (x*x,x);
    \addplot[name path = B, green, ultra thick, domain=0:1] (x,x*x);
    \addplot[green, ultra thick, domain=1:1.41] (x,x*x);
    %\addplot [gray!20, opacity=0.5] fill between [of = A and B, soft clip={}];
    \addplot[mark=*] coordinates {(1,1)}; \node [right] at (1.15,0.975) {$\{x^*\} = \ST$};
    \node [right] at (.38,.5) {$\mathcal{S}'$}; %\node [right] at (.38,1.5) {$\mathcal{S}''$};
    %\addplot[mark=*] coordinates {(0.57,1.75)}; \addplot[mark=*] coordinates {(1.32,1.75)}; \addplot[mark=*] coordinates {(1.75,1.32)};
    \node [left] at (0.58,1.75) {$\scriptstyle 4$};\node [right] at (0.58,1.75) {$\scriptstyle 5$};
    \node [left] at (1.32,1.75) {$\scriptstyle 2$};\node [right] at (1.32,1.75) {$\scriptstyle 1$};
    \node [above] at (1.75,1.32) {$\scriptstyle 2$};\node [below] at (1.75,1.32) {$\scriptstyle 3$};
\end{axis}
\end{tikzpicture}
\hfill
\blue
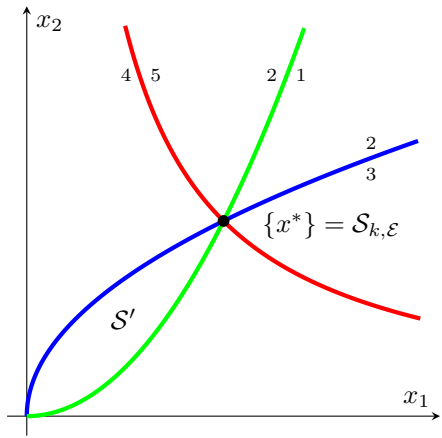
{
\begin{minipage}[b]{0.48\textwidth} \small
Finally, let $G$ have two strongly connected components $G'=(V',E')$, $G''=(V'',E'')$
with $V'=\{1,2,3\}$, $V''=\{4,5\}$, %(and arbitrary $E$), 
$y(1)$, $y(2)$, $y(3)$ as above, 
and $y(4) = {0 \choose 0}$, $y(5) = {1 \choose 1}$.
(Assume $x^* = {1 \choose 1}$, and hence $(x/x^*)^{y(4)} = x^{y(4)} = 1$, $x^{y(5)} = x_1 x_2$.)
Consider the order $x^{y(1)} \le x^{y(2)} \le x^{y(3)}$ and $x^{y(4)} \le x^{y(5)}$,
that is, $\ST$ with $\EE = \{{1 \to 2}, {2 \to 3}, {4 \to 5}\}$.
Explicitly, $\ST = \mathcal{S}' \cap \mathcal{S}''$
with $\mathcal{S}' = \{ x \mid x^{y(1)} \le x^{y(2)} \le x^{y(3)} \}$ (as above)
and $\mathcal{S}'' = \{ x \mid x^{y(4)} \le x^{y(5)} \}$
(the region on and above the red line).
%The red line specifies $x^{y(4)} = x^{y(5)}$; only above it, $x^{y(5)} > x^{y(4)}$.
As a consequence, $\ST = \{ x^* \}$ is trivial
(equals the set of complex-balanced equilibria).
In logarithmic coordinates, % $z = \ln (x/x^*)$, 
the corresponding polyhedral cone~$\CE = \{0\}$ is trivial.
%\bigskip \smallskip
\end{minipage}
}

\blue{
In general, $\ST = x^* \circ \e^{S^\perp}$ (equals the set of complex-balanced equilibria)
if and only if $\CE = S^\perp$. % with the stoichiometric subspace $S$.
In the example, $S = \R^2$ and $S^\perp = \{0\}$.
}

% ========= ========= ========= ========= ========= ========= ========= =========

\subsection{Binomial differential inclusions} \label{sec:bdi}

Finally, we extend a classical result by Horn and Jackson from 1972.
%on the asymptotic stability of complex-balanced equilibria.
%and the non-existence of other steady states.
%
\begin{theorem}[cf.~\cite{HornJackson1972}, Theorem 6A] \label{thm:cbe}
Let $(G_k,y)$ be a mass-action system
and $x^* \in \R^n_>$ be a positive CBE
of the dynamical system~\eqref{dynsysAk}. 
Then,
\[
\left( \ln \frac{x}{x^*} \right)^\trans \! f_k(x) < 0
\]
for all $x \in \R^n_>$ that are not complex-balanced equilibria.
Hence, 
(i) all positive equilibria are complex-balanced, and
(ii) $x^*$ is asymptotically stable. % (in its stoichiometric class).
\end{theorem}

All proofs are based on the entropy-like Lyapunov function $L \colon \R^n_> \to \R$,
\begin{equation}
L(x) = \sum_{i=1}^n x_i \left( \ln \frac{x_i}{x^*_i} -1 \right) + x^*_i .
\end{equation}
For $x \in \R^n_>$,
\[
L(x) \ge 0 \quad \text{ with ``='' if and only if } x = x^* ,
\]
$\nabla L = \left( \ln \frac{x}{x^*} \right)^\trans$,
and hence
\[
\dd{}{t} \, L(x(t)) = \nabla L \, \dd{x}{t} = \left( \ln \frac{x}{x^*} \right)^\trans \! f_k(x) .
\]
If $\left( \ln \frac{x}{x^*} \right)^\trans \! f_k(x) \le 0$ with ``='' if and only if $x = x^*$, then $L(x)$ is a strict Lyapunov function,
and $x^*$ is asymptotically stable.

Previous proofs further use inequalities for the exponential function or the logarithm and cycle decomposition of the graph,
cf.~\cite{HornJackson1972,Sontag2001,Anderson2011,Gopalkrishnan2014}.
For a new proof using monomial evaluation orders and corresponding geometric objects (strata and polyhedral cones), see Appendix~\ref{app:stab}.

In the following, we extend the stability result
and provide a maximally transparent, polyhedral-geometry proof.
First, we relate the dynamics in a given stratum
to the corresponding polyhedral cone.

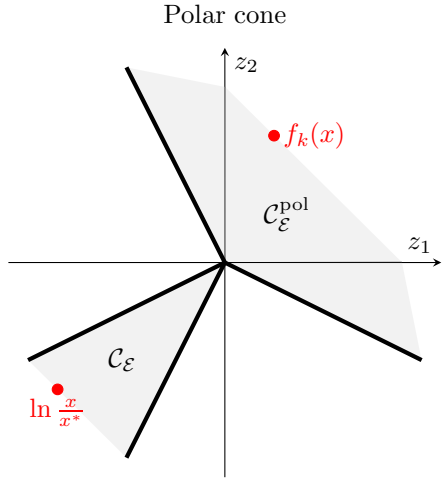
\begin{tikzpicture}
\begin{axis}[width=0.6\textwidth,height=0.6\textwidth,
    axis lines=middle,xtick=\empty,ytick=\empty,
    xmin=-1.1,xmax=1.1,ymin=-1.1,ymax=1.1,samples=10,
    xlabel={$z_1$},ylabel={$z_2$},title={Polar cone}] 
    \addplot[name path = VL, gray, ultra thin, domain=0:1](-x/1000,.9*x);
    \addplot[name path = VR, gray, ultra thin, domain=0:1](x/1000,.9*x);
    \addplot[name path = H, ultra thin, domain=0:1](.9*x,x/1000);
    \addplot[name path = A, ultra thick, domain=-1:0]{x/2};
    \addplot[name path = B, ultra thick, domain=-1:0](x/2,x);
    \addplot[name path = C, ultra thick, domain=-1:0](x/2,-x);
    \addplot[name path = D, ultra thick, domain=0:1]{-x/2};
    \addplot[gray!20, opacity=0.5] fill between [of = A and B, soft clip={}];
    \addplot[gray!20, opacity=0.5] fill between [of = C and VL, soft clip={}];
    \addplot[gray!20, opacity=0.5] fill between [of = VR and H, soft clip={}];
    \addplot[gray!20, opacity=0.5] fill between [of = H and D, soft clip={}];
    \node [right] at (-.64,-.5) {$\CE$};
    \node [right] at (.15,.25) {$\CE^\pol$};
    \addplot[mark=*,red] coordinates {(-.85,-.65)}; \node [below,red] at (-.85,-.65) {$\ln \frac{x}{x^*}$};
    \addplot[mark=*,red] coordinates {(.25,.65)}; \node [right,red] at (.25,.65) {$f_k(x)$};
\end{axis}
\end{tikzpicture}
\hfill
\begin{minipage}[b]{0.48\textwidth} \small
In Proposition~\ref{prop:emb} below, 
we use the concept of the polar cone \[C^\pol = \{ y \mid y \cdot x \le 0 \text{ for all } x \in C\}\] of a set $C$,
\blue{where} $C^\pol \subset (\lin C)^\perp$,
and $y \in \innt C^\pol$ if and only if $y \cdot x < 0$ for all $x \in C \setminus \lin C$.
In our setting, a monomial evaluation order \blue{(represented by a chain graph $G_\EE$)}
determines a stratum $\ST$ \blue{and a corresponding polyhedral cone~$\CE$ (which are both} full-dimensional). 
\blue{In particular, $\CE$} has a non-trivial lineality space $\lin \CE = S^\perp$ if \blue{and only if} $S \neq \R^n$,
and $\CE^\pol \subset S$.
By Proposition~\ref{prop:emb},
if $\ln \frac{x}{x^*} \in \CE$, then $f_k(x) \in \CE^\pol$.
%\medskip
\end{minipage}

\begin{proposition} \label{prop:emb}
Let $(G_k,y)$ be a complex-balanced mass-action system,
\blue{$G_\EE$ be a chain graph,}
and $\ST \subset \R^n_>$ be a stratum. 
%(for some chain graph $G_\EE=(V,\EE)$).
Then, for all $x \in \ST$ that are not positive complex-balanced equilibria,
$f_k(x) \in \innt \CE^\pol$.
%
%and all $u \in \CE$ that are not in the lineality space,
%\[
%u^\trans f_k(x) < 0 .
%\]
%That is, 
%$f_k(x) \in \innt \CE^\pol$ for all $x \in \ST$ that are not positive complex-balanced equilibria.
%(for all $x \in \ST \setminus (x^* \circ \e^{S^\perp})$, that is, for all $\ln \frac{x}{x^*} \in \CE \setminus S^\perp$.)
\end{proposition}
\begin{proof}
Let $x \in \ST$ and $u \in \CE$. % for some chain graph $G_\EE=(V,\EE)$.
Using the dynamical system~\eqref{dynsysAk} and Theorem~\ref{thm:Ac},
we have
\begin{align*}
u^\trans \! f_k(x)
&= 
u^\trans Y A_k \, x^Y \\
&= - u^\trans Y I_\EE \Ac I_\EE^\trans \diag(K_k^{-1}) \, x^Y \\
&= - \, a^\trans \Ac \, b
\end{align*}
with
\begin{align*}
a(u) &= (Y I_\EE)^\trans u , \\
b(x) &= I_\EE^\trans \diag(K_k^{-1}) \, x^Y .
\end{align*}
Using $\ST$ and $\CE$ as in Eqns.~\eqref{st1} and \eqref{ce}, we have $b \ge 0$ and $a \ge 0$.

By Theorem~\ref{thm:Ac}, the \blue{\irr matrix of the graph Laplacian}, $\Ac \in \R^{\EE \times \EE}$, is non-negative with positive diagonal.
Hence,
\[
u^\trans \! f_k(x) 
= - a^\trans \Ac \, b \le 0 
\]
and
\[
f_k(x) \in \CE^\pol .
\]
Recall $f_k(x) \in S$. Hence, $u^\trans \! f_k(x) = 0$ for $u \in \lin \CE = S^\perp$.
So, let $x \in \ST$ not be a CBE, that is, $f_k(x) \neq 0$,
and $u \in \CE$ not lie in the lineality space, that is, $u \not\in S^\perp$.
Then, $u^\trans \! f_k(x) \neq 0$. Altogether, $u^\trans \! f_k(x) < 0$ and $f_k(x) \in \innt \CE^\pol$.
%Now, let $x$ not be complex-balanced. 
%Then, $b \neq 0$.
%%
%\begin{itemize}
%\item
%If $u \in \innt \CE$, then $a>0$ and hence $u^\trans \! f_k(x) = - a^\trans \Ac \, b < 0$.
%\item
%If $u \not\in \innt \CE$, then 
%\[
%u \in \bigcap_{\EE \colon u \in \CE} \CE 
%\]
%with $|\{\EE \colon u \in \CE\}|>1$.
%On the one hand,
%$u$ lies on the boundary of more than one stratum;
%on the other hand,
%$u$ lies in the interior of all these strata,
%\[
%u \in \innt \left( \bigcup_{\EE \colon u \in \CE} \CE \right) .
%\]
%Hence,
%\[
%f_k(x) \in \bigcap_{\EE \colon u \in \CE} \CE^\pol  = \left( \bigcup_{\EE \colon u \in \CE} \CE \right)^\pol 
%\]
%and $u^\trans \! f_k(x) < 0$. %\qed
%\end{itemize}
\end{proof}

Now, let $(G_k,y)$ be a mass-action system
and $x^* \in \R^n_>$ be a positive CBE
of the dynamical system~\eqref{dynsysAk}. 
Proposition~\ref{prop:emb} suggests to introduce a corresponding piece-wise constant {\em binomial differential inclusion} as
\begin{equation*}
\dd{x}{t} \in F_{x^*}(x) = 
\begin{cases}
\{0\} & \text{ for } x \in x^* \circ \e^{S^\perp} , \\
\innt \left( \bigcap_{\EE \colon x \in \ST} \CE^\pol \right) & \text{ for } x \not\in x^* \circ \e^{S^\perp} ,
\end{cases}
\end{equation*}
thereby explicitly specifying the set of positive equilibria~$x^* \circ \e^{S^\perp}$.
Equivalently, using $x \in \ST$ $\Leftrightarrow$ $\ln \frac{x}{x^*} \in \CE$, 
\begin{equation} \label{diffincl} 
\dd{x}{t} \in F( \textstyle \ln \frac{x}{x^*})
\quad \text{with} \quad
F(u) = 
\begin{cases}
\{0\} & \text{ for } u \in S^\perp , \\
\innt \left( \bigcap_{\EE \colon u \in \CE} \CE^\pol \right) & \text{ for } u \not\in S^\perp .
\end{cases}
\end{equation}

Proposition~\ref{prop:emb} immediately implies the following result.
\begin{theorem}
Let $(G_k,y)$ be a mass-action system
and $x^* \in \R^n_>$ be a positive CBE
of the dynamical system~\eqref{dynsysAk}. 
Then, the mass-action system can be embedded in the binomial differential inclusion~\eqref{diffincl}.
\end{theorem}

Finally, we extend Theorem~\ref{thm:cbe} (from complex-balanced mass-action systems to binomial differential inclusions).
\begin{theorem} \label{thm:bdi}
Let $x^* \in \R^n_>$ be a positive equilibrium of the binomial differential inclusion~\eqref{diffincl}.
Then,
\[
\left( \ln \frac{x}{x^*} \right)^\trans \! f < 0 ,
\]
for all $x \in \R^n_>$ that are not positive equilibria and all $f \in F(\ln \frac{x}{x^*})$.
Hence, $x^*$ is asymptotically stable.
\end{theorem}
\begin{proof}
Let $\ST \subset \R^n_>$ be a stratum and $x \in \ST$ not be a \blue{positive equilibrium.}
% for some chain graph $G_\EE=(V,\EE)$.
On the one hand,
\[ x \in \ST \setminus (x^* \circ \e^{S^\perp}), \quad \text{that is,} \quad 
\ln \frac{x}{x^*} \in \CE \setminus S^\perp .
\]
that is, $\ln \frac{x}{x^*}$ lies in $\CE$, but not in the lineality space $\lin \CE = S^\perp$.
On the other hand,
\[
f \in F\left(\ln \frac{x}{x^*}\right) = \innt \left( \bigcap_{\EE \colon \ln \! \frac{x}{x^*} \in \CE} \CE^\pol \right) ,
\]
and $\CE^\pol \subset S$.
Hence $\left( \ln \frac{x}{x^*} \right)^\trans \! f < 0$,
and $L(x)$ is a strict Lyapunov function.
\end{proof}

\begin{remark} \label{rem:wr} 
\blue{%
Even if a weakly reversible mass-action system $(G_k,y)$ does not admit a complex-balanced equilibrium $x^*$,
it can be embedded in a piece-wise constant differential inclusion.
Technically,
the absence of a CBE $x^*$ does not allow to pass from the polyhedron $\PE$ (with given monomial evaluation order)
to the cone $\CE$, cf.~Eqns.~\eqref{pe} and~\eqref{ce}.
That is, instead of a central hyperplane arrangement (that defines the cones $\CE$),
one considers a non-central hyperplane arrangement (that defines the polyhedra $\PE$).
In analogy to Proposition~\ref{prop:emb}, one can show that, for a chain graph $G_\EE$ and a stratum $\ST$,
it holds that $f_k(x) \in \innt ( \rec ( \PE )^\pol )$, for all $x \in \ST$.
Here, $\rec(C)$ denotes the {\em recession cone} of a set $C$.
}
\end{remark}

\subsection{Discussion} % and further directions}
\label{sec:dis}

\blue{As Horn and Jackson in 1972~\cite[Theorem~6A]{HornJackson1972},
we have shown that, in mass-action systems with a positive complex-balanced equilibrium,
every positive equilibrium is complex-balanced and asymptotically stable.
For a proof using the new decomposition of the graph Laplacian,
monomial evaluation orders, and corresponding geometric objects (strata and polyhedral cones), 
see Appendix~\ref{app:stab}.
In fact, we have extended the result to {\em binomial differential inclusions} (BDIs), introduced in this work.
Every positive equilibrium of a BDI is asymptotically stable, see Theorem~\ref{thm:bdi}.
}

\subsubsection*{Binomial and toric differential inclusions}

\blue{
Given a reaction network $(G,y)$ with graph $G=(V,E)$ and ``complex'' map $y \colon V \to \R^n_\ge$, a BDI} 
%A binomial differential inclusion arises from a complex-balanced mass-action system.
%%but it is not defined in terms of the underlying graph.
%To be precise, it 
depends on the components of the graph (but not on the exact edge set~$E$)
and on some \blue{positive equilibrium~$x^*$} (but not explicitly on the rate constants).
In fact, it is mainly determined by stoichiometry,
namely by \blue{pairwise} differences of complexes, \blue{defining a hyperplane arrangement.
In particular,}
monomial evaluation orders
\blue{correspond to} polyhedral cones \blue{(in logarithmic coordinates)} and strata \blue{(in the original positive variables).
More formally, a BDI is given by a hyperplane arrangement (with lineality space $S^\perp$) and a positive equilibrium $x^*$,
see Equation~\eqref{diffincl}.
Most importantly,
complex-balanced mass-action systems can be embedded in BDIs.}

Recently, {\em toric differential inclusions} (TDIs) have been used in a proposed proof~\cite{Craciun2015,Craciun2019} of the global attractor conjecture~\cite{Horn1974},
stating that complex-balanced equilibria are not just asymptotically, but also globally stable.
%In this work, we only consider local stability and do not address the GAC.
%For completeness, we briefly compare ``binomial'' and ``toric'' differential inclusions:
\blue{In fact, TDIs also allow to tackle the persistence and permanence conjectures
for (weakly reversible) mass-action systems with (time-)variable rate constants.}
In the classical setting, rate constants $k>0$ are fixed,
whereas, \blue{in the study of the conjectures mentioned above,}
rate constants $\epsilon \le k(t) \le 1/\epsilon$ may vary over time, but are bounded \blue{\cite{Anderson2011,Craciun2013}.}
%As a consequence, the set of complex-balanced equilibria and hence the regions with given monomial evaluation order may vary over time.
To address this complication, ``uncertainty regions'' \blue{with thickness $\delta(\epsilon)$} around the boundaries of ``regions with definite monomial order'' are introduced.
%In both settings,
%mass-action systems are embedded in differential inclusions
%by considering regions with {\em total} orders (represented by chain graphs, in this work).
\blue{On the one hand, BDIs are special cases of TDIs with $\delta \to 0$ (modulo a translation of the hyperplane arrangement by $\log x^*$),
and also the piece-wise constant differential inclusions %(for weakly reversible mass-action systems)
mentioned in Remark~\ref{rem:wr} can be embedded in TDIs (with $\delta>0$).
On the other hand, 
BDIs allow to consider (the asymptotic stability of) positive equilibria,
%whereas TDIs capture the dynamics close to the boundary of the positive orthant and ``forget'' about equilibria.
%whereas TDIs capture the dynamics of mass-action systems without being explicit about the the stability of steady states.
whereas TDIs capture the dynamics close to the boundary of the positive orthant
without being explicit about equilibria.
}

%To date, it is not clear whether binomial differential inclusions can be embedded in toric differential inclusions.
%If this is the case (and toric differential inclusions are persistent~\cite{Craciun2015}),
%then all positive equilibria of binomial differential inclusions are globally stable.

\subsubsection*{Generalized mass-action systems}

In previous work, we have studied {\em generalized} mass-action systems~\cite{MuellerRegensburger2012,MuellerRegensburger2014,Mueller2016,MuellerHofbauerRegensburger2019,CraciunMuellerPanteaYu2019,BorosMuellerRegensburger2020}.
\blue{In order to motivate the setting, we consider the reaction ${1 \mathsf{X}_1+ 1 \mathsf{X}_2 \to \mathsf{X}_3}$
with ``stoichiometric'' coefficients equal to 1.
Under the assumption of generalized mass-action kinetics,
its rate is given by $v = k \, (x_1)^a (x_2)^b$ with arbitrary ``kinetic orders'' $a,b > 0$ (in particular, different from 1).
Using the complexes $y = (1,1,0,0,\ldots)^\trans$, $y'=(0,0,1,0,\ldots)^\trans$,
and the kinetic-order complex $\tilde y = (a,b,0,0, \ldots)^\trans$,
we can write the reaction as $y \to y'$ with rate $v = k \, x^{\tilde y}$.
For a network, the resulting dynamical system,
\begin{equation}
\dd{x}{t} = Y A_k \, x^{\tilde Y} ,
\end{equation}
is determined by the matrices $Y$ (by stoichiometry), $\tilde Y$ (by kinetics), and $A_k$ (by a graph).
For generalized mass-action systems,
asymptotic} stability of complex-balanced equilibria and non-existence of other steady states are not guaranteed
(as for classical mass-action systems,
cf.~Theorem~\ref{thm:cbe}).
We have already provided necessary conditions for linear stability of complex-balanced equilibria~\cite{BorosMuellerRegensburger2020}.
In parallel work~\cite{MuellerRegensburger2022}, we use the new decomposition of the graph Laplacian and monomial evaluation orders 
to study sufficient conditions for linear stability of complex-balanced equilibria
and non-existence of other steady states.
%\blue{In future work,
%we plan to use ``generalized binomial differential inclusions'' (involving two hyperplane arrangements, arising from $Y$ and $\tilde Y$, respectively)
%in order to study asymptotic stability of complex-balanced equilibria.
%}

% ========= ========= ========= ========= ========= ========= ========= =========

\clearpage

\subsection*{Acknowledgments}

We thank Georg Regensburger and Bal\'azs Boros for fruitful discussions
based on the first version of this manuscript, %(from 2015),
in particular, on monomial evaluation orders and the binomial structure of mass-action systems (with GR)
and on doubly stochastic matrices (with BB).
\blue{Further, we thank Abhishek Deshpande for clarifying discussions on toric differential inclusions
and two anonymous reviewers for their very careful reading and their many helpful comments.}

SM was supported by the Austrian Science Fund (FWF), project P33218-N.

\subsection*{Data availability}

Data sharing is not applicable to this article as no datasets were generated or analysed during the current study.

\subsection*{Conflict of interest}

The author declares that there is no conflict of interest.

% ========= ========= ========= ========= ========= ========= ========= =========

\bibliographystyle{abbrv}
\bibliography{fractional,MR}

\begin{thebibliography}{10}

\bibitem{Anderson2011}
D.~Anderson.
\newblock A proof of the global attractor conjecture in the single linkage
  class case.
\newblock {\em {SIAM} Journal on Applied Mathematics}, 71(4):1487--1508, 2011.

\bibitem{Anderson2014}
D.~Anderson.
\newblock A short note on the {L}yapunov function for complex-balanced chemical
  reaction networks.
\newblock {\em Unpublished}, 2014.
\newblock
  \href{https://people.math.wisc.edu/~dfanderson/CRNT_Lyapunov.pdf}{CRNT\_Lyapunov.pdf}.

\bibitem{BelkinNiyogi2003}
M.~Belkin and P.~Niyogi.
\newblock Laplacian eigenmaps for dimensionality reduction and data
  representation.
\newblock {\em Neural computation}, 15(6):1373--1396, 2003.

\bibitem{Birch1963}
M.~W. Birch.
\newblock Maximum likelihood in three-way contingency tables.
\newblock {\em J. Roy. Statist. Soc. Ser. B}, 25:220--233, 1963.

\bibitem{Birkhoff1946}
G.~Birkhoff.
\newblock Tres observaciones sobre el algebra lineal.
\newblock {\em Universidad Nacional de Tucum\'{a}n. Revista. Serie A},
  5:147--151, 1946.

\bibitem{BorosMuellerRegensburger2020}
B.~Boros, S.~M{\"u}ller, and G.~Regensburger.
\newblock Complex-balanced equilibria of generalized mass-action systems:
  Necessary conditions for linear stability.
\newblock {\em Mathematical Biosciences and Engineering}, 17(1):442--459, 2020.

\bibitem{Craciun2015}
G.~Craciun.
\newblock Toric differential inclusions and a proof of the global attractor
  conjecture.
\newblock {\em arXiv}, 2015.
\newblock \href{https://arxiv.org/abs/1501.02860}{arXiv:1501.02860} [math.DS].

\bibitem{Craciun2019}
G.~Craciun.
\newblock Polynomial dynamical systems, reaction networks, and toric
  differential inclusions.
\newblock {\em SIAM J. Appl. Algebra Geom.}, 3(1):87--106, 2019.

\bibitem{CraciunDickensteinShiuSturmfels2009}
G.~Craciun, A.~Dickenstein, A.~Shiu, and B.~Sturmfels.
\newblock Toric dynamical systems.
\newblock {\em J. Symbolic Comput.}, 44:1551--1565, 2009.

\bibitem{CraciunMuellerPanteaYu2019}
G.~{Craciun}, S.~{M{\"u}ller}, C.~{Pantea}, and P.~{Yu}.
\newblock {A generalization of Birch's theorem and vertex-balanced steady
  states for generalized mass-action systems}.
\newblock {\em Mathematical Biosciences and Engineering}, 16(6):8243--8267,
  2019.

\bibitem{Craciun2013}
G.~Craciun, F.~Nazarov, and C.~Pantea.
\newblock Persistence and permanence of mass-action and power-law dynamical
  systems.
\newblock {\em {SIAM} Journal on Applied Mathematics}, 73(1):305--329, 2013.

\bibitem{Feinberg1972}
M.~Feinberg.
\newblock Complex balancing in general kinetic systems.
\newblock {\em Arch. Rational Mech. Anal.}, 49:187--194, 1972/73.

\bibitem{FeinbergHorn1977}
M.~Feinberg and F.~J.~M. Horn.
\newblock Chemical mechanism structure and the coincidence of the
  stoichiometric and kinetic subspaces.
\newblock {\em Arch. Rational Mech. Anal.}, 66(1):83--97, 1977.

\bibitem{Fulton1993}
W.~Fulton.
\newblock {\em Introduction to toric varieties}, volume 131 of {\em Ann. of
  Math. Stud.}
\newblock Princeton University Press, Princeton, NJ, 1993.

\bibitem{Gopalkrishnan2014}
M.~Gopalkrishnan.
\newblock On the {L}yapunov function for complex-balanced mass-action systems.
\newblock {\em arXiv}, 2014.
\newblock \href{https://arxiv.org/abs/1312.3043}{arXiv:1312.3043} [math.DS].

\bibitem{GopalkrishnanMillerShiu2014}
M.~Gopalkrishnan, E.~Miller, and A.~Shiu.
\newblock A geometric approach to the global attractor conjecture.
\newblock {\em SIAM J. Appl. Dyn. Syst.}, 13:758--797, 2014.

\bibitem{Gunawardena2012}
J.~Gunawardena.
\newblock A linear framework for time-scale separation in nonlinear biochemical
  systems.
\newblock {\em {PLoS} {ONE}}, 7(5):e36321, 2012.

\bibitem{Horn1972}
F.~Horn.
\newblock Necessary and sufficient conditions for complex balancing in chemical
  kinetics.
\newblock {\em Arch. Rational Mech. Anal.}, 49:172--186, 1972/73.

\bibitem{Horn1974}
F.~Horn.
\newblock The dynamics of open reaction systems.
\newblock In {\em Mathematical aspects of chemical and biochemical problems and
  quantum chemistry ({P}roc. {SIAM}-{AMS} {S}ympos. {A}ppl. {M}ath., {N}ew
  {Y}ork, 1974)}, pages 125--137. SIAM--AMS Proceedings, Vol. VIII. Amer. Math.
  Soc., Providence, R.I., 1974.

\bibitem{HornJackson1972}
F.~Horn and R.~Jackson.
\newblock General mass action kinetics.
\newblock {\em Arch. Rational Mech. Anal.}, 47:81--116, 1972.

\bibitem{Kandori1993}
M.~Kandori, G.~J. Mailath, and R.~Rob.
\newblock Learning, mutation, and long run equilibria in games.
\newblock {\em Econometrica}, 61(1):29--56, 1993.

\bibitem{Merris1994}
R.~Merris.
\newblock Laplacian matrices of graphs: a survey.
\newblock {\em Linear Algebra Appl.}, 197/198:143--176, 1994.
\newblock Second Conference of the International Linear Algebra Society (ILAS)
  (Lisbon, 1992).

\bibitem{MirzaevGunawardena2013}
I.~Mirzaev and J.~Gunawardena.
\newblock Laplacian dynamics on general graphs.
\newblock {\em Bull. Math. Biol.}, 75(11):2118--2149, 2013.

\bibitem{Mohar1991}
B.~Mohar.
\newblock The {L}aplacian spectrum of graphs.
\newblock In {\em Graph theory, combinatorics, and applications. {V}ol. 2
  ({K}alamazoo, {MI}, 1988)}, Wiley-Intersci. Publ., pages 871--898. Wiley, New
  York, 1991.

\bibitem{Mueller2016}
S.~M{\"u}ller, E.~Feliu, G.~Regensburger, C.~Conradi, A.~Shiu, and
  A.~Dickenstein.
\newblock Sign conditions for injectivity of generalized polynomial maps with
  applications to chemical reaction networks and real algebraic geometry.
\newblock {\em Found. Comput. Math.}, 16(1):69--97, 2016.

\bibitem{MuellerHofbauerRegensburger2019}
S.~M{\"u}ller, J.~Hofbauer, and G.~Regensburger.
\newblock On the bijectivity of families of exponential/generalized polynomial
  maps.
\newblock {\em SIAM J. Appl. Algebra Geom.}, 3(3):412--438, 2019.

\bibitem{MuellerRegensburger2012}
S.~M\"uller and G.~Regensburger.
\newblock Generalized mass action systems: {C}omplex balancing equilibria and
  sign vectors of the stoichiometric and kinetic-order subspaces.
\newblock {\em SIAM J. Appl. Math.}, 72(6):1926--1947, 2012.

\bibitem{MuellerRegensburger2014}
S.~M\"uller and G.~Regensburger.
\newblock Generalized mass-action systems and positive solutions of polynomial
  equations with real and symbolic exponents.
\newblock In V.~P. Gerdt, W.~Koepf, E.~W. Mayr, and E.~H. Vorozhtsov, editors,
  {\em Computer Algebra in Scientific Computing. Proceedings of the 16th
  International Workshop (CASC 2014)}, volume 8660 of {\em Lecture Notes in
  Comput. Sci.}, pages 302--323, Berlin/Heidelberg, 2014. Springer.

\bibitem{MuellerRegensburger2022}
S.~M{\"u}ller and G.~Regensburger.
\newblock Sufficient conditions for linear stability of complex-balanced
  equilibria in generalized mass-action systems.
\newblock {\em arXiv}, 2022.
\newblock \href{https://arxiv.org/abs/2212.11039}{arXiv:2212.11039} [math.DS].

\bibitem{PachterSturmfels2005}
L.~Pachter and B.~Sturmfels.
\newblock Statistics.
\newblock In {\em Algebraic statistics for computational biology}, pages 3--42.
  Cambridge Univ. Press, New York, 2005.

\bibitem{Shuman2013}
D.~I. Shuman, S.~K. Narang, P.~Frossard, A.~Ortega, and P.~Vandergheynst.
\newblock The emerging field of signal processing on graphs: Extending
  high-dimensional data analysis to networks and other irregular domains.
\newblock {\em IEEE Signal Processing Magazine}, 30(3):83--98, 2013.

\bibitem{SiegelJohnston2011}
D.~Siegel and M.~D. Johnston.
\newblock A stratum approach to global stability of complex balanced systems.
\newblock {\em Dyn. Syst.}, 26(2):125--146, 2011.

\bibitem{Sontag2001}
E.~D. Sontag.
\newblock Structure and stability of certain chemical networks and applications
  to the kinetic proofreading model of {T}-cell receptor signal transduction.
\newblock {\em IEEE Trans. Automat. Control}, 46(7):1028--1047, 2001.

\bibitem{Tutte1948}
W.~T. Tutte.
\newblock The dissection of equilateral triangles into equilateral triangles.
\newblock {\em Proc. Cambridge Philos. Soc.}, 44:463--482, 1948.

\bibitem{vonNeumann1953}
J.~von Neumann.
\newblock A certain zero-sum two-person game equivalent to the optimal
  assignment problem.
\newblock In {\em Contributions to the theory of games, vol. 2}, Annals of
  Mathematics Studies, no. 28, pages 5--12. Princeton University Press,
  Princeton, N.J., 1953.

\end{thebibliography}

% ========= ========= ========= ========= ========= ========= ========= =========

\clearpage

\appendix

\section*{Appendix} \label{app}

\section{\texorpdfstring{Explicit formulas for $K_k$ and $A_k \diag(K_k)$}{}} 
\label{app:formulas} 

We consider a strongly connected, labeled, simple digraph $G_k=(V,E,k)$.
Based on the underlying unlabeled graph $G=(V,E)$, we introduce three sets of subgraphs.

\begin{enumerate}
\item
For $i \in V$, we introduce the set $T_i$ of subgraphs of $G$ that fulfill two requirements:
(i) a subgraph does not contain a cycle,
and (ii) every vertex except $i$ is the source of exactly one edge.

(This is the set of directed spanning trees of $G$ rooted at vertex $i$ 
and directed towards the root.)
\item
For $i \in V$, we introduce the set $G_i$ of subgraphs of $G$ that fulfill three requirements:
(i) a subgraph contains exactly one cycle, (ii) this cycle contains vertex $i$, 
and (iii) every vertex is the source of exactly one edge. 

(This set has been used in~\cite[Lemma~1]{Kandori1993}.)
\item
For a cycle $C$ contained in $G$, we introduce the set $G_C$ of subgraphs of $G$ that fulfill three requirements:
(i) a subgraph contains cycle $C$, (ii) $C$ is the only cycle, 
and (iii) every vertex is the source of exactly one edge. 

\blue{(This is the obvious extension from fixing a vertex to fixing a cycle.)}
\end{enumerate}

In the following, we write $(i \to i') \in G$ short for $(i \to i') \in E$, where $G=(V,E)$.

\blue{
\begin{fact} 
\[
A_k K_k = 0 , % \tag{A1}
\]
where
\[
(K_k)_i = \sum_{T \in T_i} \; \prod_{(j \to j') \in T} k_{j \to j'} , \quad \text{for } i \in V . %\tag{A2}
\]
\end{fact}
}

\begin{proof}
Recall
\[
(A_k \, \psi)_i = \sum_{(i' \to i) \in G} k_{i' \to i} \, \psi_{i'} - \sum_{(i \to i') \in G} k_{i \to i'} \, \psi_i \quad \text{for } i \in V .
\]
%and
%\begin{equation*}
%(K_k)_i = \sum_{T \in T_i} \; \prod_{(j \to j') \in T} k_{j \to j'} , \quad \text{for } i \in V . \tag{A2}
%\end{equation*}
On the one hand, 
every subgraph $S \in G_i$ gives rise to a spanning tree $T \in T_i$ and vice versa
(by removing/adding the edge $i \to i'$).
For $i \in V$,
\begin{align*}
\sum_{S \in G_i} \prod_{(j \to j') \in S} k_{j \to j'} 
&= \sum_{T \in T_i} \sum_{(i \to i') \in G} \prod_{(j \to j') \in T} k_{j \to j'} \cdot k_{i \to i'} \\
&= \sum_{(i \to i') \in G} k_{i \to i'} \sum_{T \in T_i} \prod_{(j \to j') \in T} k_{j \to j'} \\
&= \sum_{(i \to i') \in G} k_{i \to i'} \, (K_k)_i .
\end{align*}
On the other hand,
every subgraph $S \in G_i$ gives rise to a spanning tree $T \in T_{i'}$ and vice versa (by removing/adding the edge $i' \to i$ \blue{that is} in the cycle).
For $i \in V$,
\begin{align*}
\sum_{S \in G_i} \prod_{(j \to j') \in S} k_{j \to j'} 
&= \sum_{T \in T_{i'}} \sum_{(i' \to i) \in G} \prod_{(j \to j') \in T} k_{j \to j'} \cdot k_{i' \to i} \\
&= \sum_{(i' \to i) \in G} k_{i' \to i} \sum_{T \in T_{i'}} \prod_{(j \to j') \in S} k_{j \to j'} \\
&= \sum_{(i' \to i) \in G} k_{i' \to i} \, (K_k)_{i'} .
\end{align*}
Hence,
\[
\sum_{(i' \to i) \in G} k_{i' \to i} \, (K_k)_{i'} - \sum_{(i \to i') \in G} k_{i \to i'} \, (K_k)_i = 0 \quad \text{for } i \in V ,
\]
that is, $\psi = K_k$ solves $A_k \, \psi = 0$.
\end{proof}

\blue{
\begin{fact}
%Second, we show
\[
A_k \diag(K_k) = \sum_{C} \la_{k,C} \, A_C , %\tag{A3}
\]
where the sum is over all cycles $C$ contained in $G$,
\[
\la_{k,C} = \sum_{S \in G_C} \; \prod_{(j \to j') \in S} k_{j \to j'} , %\tag{A4}
\]
and $A_C$ is the Laplacian matrix of the cycle $C$ with $k=\bar 1 \in \R^E_>$ (all edge labels set to 1).
\end{fact}
}
\begin{proof}
Both matrices, $A_k \diag(K_k)$ and $\sum_{C} \la_{k,C} \, A_C$, have zero row and column sums.
Hence, it is sufficient to compare the off-diagonal entries.

On the one hand,
every spanning tree in $T_i$ gives rise to a subgraph in $G_C$ that contains the edge $i \to i'$ \blue{in the cycle} and vice versa
(by adding/removing the edge $i \to i'$). For $i \neq i'$,
\begin{align*}
(A_k \diag(K_k))_{i',i} &= k_{i \to i'} (K_k)_i \\
&= \sum_{T \in T_i} \; \prod_{(j \to j') \in T} k_{j \to j'} \cdot k_{i \to i'} \\
&= \sum_{C \colon (i \to i') \in C} \sum_{S \in G_C} \prod_{(j \to j') \in S} k_{j \to j'} \\
&= \sum_{C \colon (i \to i') \in C} \la_{k,C} .
\end{align*}
%where we write $(i \to i') \in C$ short for $(i \to i') \in E'$, where $C=(V',E')$.

On the other hand,
\[
(A_C)_{i',i} = 
\begin{cases}
1 , & \text{if } (i \to i') \in C , \\
-1 , & \text{if } i = i', \\
0 , & \text{otherwise.}
\end{cases}
\]
Hence,
\begin{align*}
\left( \sum_{C} \la_{k,C} \, A_C \right)_{i',i} &= \sum_{C \colon (i \to i') \in C} \la_{k,C} ,
\end{align*}
and the two matrices, $A_k \diag(K_k)$ and $\sum_{C} \la_{k,C} \, A_C$, agree.
\end{proof}

{\bf Remark.}
In a time-discrete, linear process $\psi' = B_k \, \psi$ with
\[
(B_k)_{i,j} = 
\begin{cases}
k_{j \to i} , & \text{if } (j \to i) \in E , \\
1-\sum_{(i \to i') \in E} k_{i \to i'} , & \text{if } i = j, \\
0 , & \text{otherwise,}
\end{cases}
\]
the edge labels $k \in \R^n_>$ do not represent transition rates,
but transition probabilities.
Then, $\sum_{(i \to i') \in E} k_{i \to i'} \le 1$, and $B_k$ is simply the matrix of transition probabilities 
with ``$k_{i \to i}$''$=1 - \sum_{(i \to i') \in E} k_{i \to i'}$ and column sums equal to one.
That is, $B_k = A_k + \mathrm{I}$, the identity matrix.
Obviously, $\psi = B_k \, \psi$ if and only if $A_k \, \psi = 0$.
Whereas $A_k \diag(K_k)$ always has zero row and column sums,
$B_k$ may (or may not) be doubly stochastic (have column {\em and} row sums equal to one).

The Birkhoff/von Neumann Theorem~\cite{Birkhoff1946,vonNeumann1953}
states that every doubly stochastic (d.s.) matrix $B \in \R^{n \times n}_\ge$ is the convex sum of permutation matrices;
however, this decomposition is not unique.
In fact, there are $n!$ permutation matrices.
Still, the polytope of d.s.\ matrices lies in an $(n-1)^2$-dimensional affine subspace of $\R^{n \times n}_\ge$,
and hence every d.s.\ matrix can be written as the sum of at most $(n-1)^2+1$ permutation matrices.

On the contrary, the matrix $A_k \diag(K_k)$ is the {\em unique} sum of {\em all} Laplacian matrices of cycles.
However, there are more than $(n-1)^2+1$ cycles, in general.
%(Laplacian matrices of cycles are special cases of permutation matrices,
%but there are more than $N$ cycles, in general.)

% ========= ========= ========= ========= ========= ========= ========= =========

\newcommand{\ddd}{-\hspace{-2ex}-\;}

\blue{
\section{Auxiliary graph-theoretic results}
\label{app:aux}
\begin{lemma}[cf.~\cite{FeinbergHorn1977}, Lemma~2] \label{lem:aux1}
Let $G_k=(V,E,k)$ be a connected, labeled, simple digraph with one absorbing strong component,
and $A_k$ and $I_E$ be the corresponding Laplacian and incidence matrices.
Then, 
\[
\im(A_k) = \im(I_E) .
\]
\end{lemma}
\begin{proof}
From graph theory,
we know that $\dim \im(I_E) = |V|-1$ and $\ker (A_k) = \im \xi$,
where $\xi \in \R^V_\ge$ has support on the absorbing strong component of $G$.
Hence, also $\dim \im(A_k) = |V|-1$.
By definition, $\im (A_k) \subseteq \im (I_E)$ and hence $\im(A_k) = \im(I_E)$.
\end{proof}
\begin{lemma}[cf.~\cite{MuellerRegensburger2014}, Proposition~5] \label{lem:aux2}
Let $G=(V,E)$ be a connected, simple digraph,
$G_\EE = (V,\EE)$ be an auxiliary digraph,
and $I_E$ and $I_\EE$ be the corresponding incidence matrices.
Then, 
\[
\im(I_\EE) = \im(I_E) .
\]
\end{lemma}
\begin{proof}
From graph theory and the definition of an auxiliary graph,
we know that $\dim \im(I_E) = \dim \im(I_\EE) = |V|-1$.
In the rest of the proof,
we show that $\im(I_E) \subseteq \im(I_\EE)$.
We consider the edge $(i \to j) \in E$ and the corresponding column $e^j-e^i$ of $I_E$,
where $e^i$ denotes the $i$th standard basis vector in $\R^V$.
Since $G_\EE$ is a directed tree,
there is a path from $i$ to $j$ in the undirected version of $G_\EE$,
that is, $i = i_1 \ddd i_2 \ddd \ldots \ddd i_l = j$ with either $(i_k \to i_{k+1}) \in \EE$ or $(i_k \leftarrow i_{k+1}) \in \EE$ for $k=1,\ldots,l-1$.
Hence,
\[
e^j-e^i = \sum_{k=1}^{l-1} \alpha_k \left( e^{i_{k+1}} - e^{i_k} \right) ,
\]
where $\alpha_k \in \{-1, 1\}$ and $e^{i_{k+1}} - e^{i_k}$ is the column of $I_\EE$
corresponding to either the edge $(i_k \to i_{k+1}) \in \EE$ or $(i_k \leftarrow i_{k+1}) \in \EE$.
\end{proof}
}

% ========= ========= ========= ========= ========= ========= ========= =========

\section{A proof of Theorem~\ref{thm:cbe}}
\label{app:stab}

We provide a proof of Theorem~\ref{thm:cbe} in the main text,
based on the entropy-like Lyapunov function.
Previous proofs further use inequalities for the exponential function or the logarithm and cycle decomposition of the graph,
cf.~\cite{HornJackson1972,Sontag2001,Anderson2014,Gopalkrishnan2014}.
We use monomial evaluation orders and corresponding geometric objects (strata and polyhedral cones).

%The new decomposition of the graph Laplacian allows to consider
%regions in the positive orthant (polyhedral cones in logarithmic coordinates)
%with given monomial evaluation orders.
%In particular, they allow a polyhedral-geometry proof of classical results by Horn and Jackson (1972)
%on the asymptotic stability of complex-balanced equilibria and the non-existence of other steady states.
%The new proof uses neither inequalities for the exponential function or the logarithm
%nor a cycle decomposition of the underlying graph, cf.~\cite{HornJackson1972,Sontag2001,Gopalkrishnan2014}.
%%but a ``sum of squares''.

\smallskip
{\bf Theorem.}
{\em 
Let $(G_k,y)$ be a mass-action system
and $x^* \in \R^n_>$ be a positive CBE
of the dynamical system~\eqref{dynsysAk}. 
Then,
\[
\left( \ln \frac{x}{x^*} \right)^\trans \! f_k(x) < 0
\]
for all $x \in \R^n_>$ that are not complex-balanced equilibria.
Hence, 
(i) all positive equilibria are complex-balanced, and
(ii) $x^*$ is asymptotically stable. % (in its stoichiometric class).
}
\begin{proof}
Let $x \in \R^n_>$ not be a CBE.
Then there is a full-dimensional subset (a stratum) $\ST \subset \R^n_>$
for some chain graph $G_\EE=(V,\EE)$ such that $x \in \ST$, that is, $\ln \frac{x}{x^*} \in \CE$.

Using the dynamical system~\eqref{dynsysAk} and Theorem~\ref{thm:Ac},
we have
\begin{align*}
\left( \ln \frac{x}{x^*} \right)^\trans \! f_k(x)
&= 
\left( \ln \frac{x}{x^*} \right)^\trans Y A_k \, x^Y \\
&= - \left( \ln \frac{x}{x^*} \right)^\trans Y I_\EE \Ac I_\EE^\trans \diag(K_k^{-1}) \, x^Y \\
&= - \, a^\trans \Ac \, b
\end{align*}
with
\begin{align*}
a(x) &= (Y I_\EE)^\trans \ln \frac{x}{x^*} , \\
b(x) &= I_\EE^\trans \diag(K_k^{-1}) \, x^Y .
\end{align*}
Using $\ST$ and $\CE$ as in Eqns.~\eqref{st1} and \eqref{ce}, we have $b \ge 0$ and $a \ge 0$.
%where $a,b \in \R^\EE$.

Since $x$ is not be a CBE,
$b \neq 0$,
that is,
there is $i \to i' \in \EE$ such that
\[
b_{i \to i'} = \frac{x^{y(i')}}{(K_k)_{i'}} - \frac{x^{y(i)}}{(K_k)_i} > 0 .
\]
By complex balancing~\eqref{cbe},
\[
\left(\frac{x}{x^*}\right)^{y(i')} - \left(\frac{x}{x^*}\right)^{y(i)} > 0
\]
and hence also
\[
a_{i \to i'} = (y(i')-y(i))^\trans \ln \frac{x}{x^*} > 0 .
\]

By Theorem~\ref{thm:Ac}, the \blue{\irr matrix of the graph Laplacian}, $\Ac \in \R^{\EE \times \EE}$ is non-negative with positive diagonal.
Hence,
\[
\left( \ln \frac{x}{x^*} \right)^\trans \! f_k(x) 
= - a^\trans \Ac \, b < 0 .
\]

(i) If there is a positive equilibrium $x \in \R^n_>$ that is not complex-balanced,
then $f_k(x)=0$, contradicting $\left( \ln \frac{x}{x^*} \right)^\trans \! f_k(x) < 0$.

%(ii) By Lemma~\ref{unique} and (i), $x^*$ is the unique positive equilibrium in its stoichiometric class.
%Now, for $x \in (x^* + S) \cap \R^n_{>0}$ with $x \neq x^*$,
%$\left( \ln \frac{x}{x^*} \right)^\trans \! f_k(x) < 0$,
%and hence $\dd{}{t} V(x(t)) < 0$ for $x(t) \in (x^* + S) \cap \R^n_{>0}$ and $x(t) \neq x^*$.

(ii) Recall that a positive CBE $x^*$ is the unique steady state in its stoichiometric compatibility class (forward invariant set).
Hence, 
\[
\dd{}{t} \, L(x(t)) = \left( \ln \frac{x}{x^*} \right)^\trans \! f_k(x) \le 0 \quad \text{ with ``='' if and only if } x=x^* ,
\]
and $L(x)$ is a strict Lyapunov function. % at $x^*$.
\end{proof}

% ========= ========= ========= ========= ========= ========= ========= =========

\end{document}